\documentclass[11pt]{article}
\usepackage[T1]{fontenc}
\usepackage{latexsym,amssymb,amsmath,amsfonts,amsthm}
\usepackage{graphics}
\usepackage{graphicx}
\usepackage{mathrsfs}
\usepackage{subfigure}
\usepackage{color}
\usepackage{epstopdf}
\newcommand{\R}{{\mathbb R}}
\newcommand{\Z}{{\mathbb Z}}

\newcommand{\no}{\nonumber}
\newcommand{\be}{\begin{eqnarray}}
\newcommand{\ben}{\begin{eqnarray*}}
\newcommand{\en}{\end{eqnarray}}
\newcommand{\enn}{\end{eqnarray*}}
\newcommand{\ba}{\backslash}
\newcommand{\pa}{\partial}

\newcommand{\ov}{\overline}

\newcommand{\ga}{\gamma}
\newcommand{\G}{\Gamma}

\newcommand{\Om}{\Omega}

\newcommand{\ra}{\rightarrow}
\newtheorem{theorem}{Theorem}[section]
\newtheorem{lemma}[theorem]{Lemma}

\newtheorem{remark}[theorem]{Remark}

\newtheorem{example}{Example}[section]

\definecolor{hgh}{rgb}{0,0,0}
\definecolor{xxx}{rgb}{0,0,0}

\begin{document}
\renewcommand{\theequation}{\arabic{section}.\arabic{equation}}
\begin{titlepage}

\title{An inverse obstacle problem with a single pair of Cauchy data: Laplace's equation case}

\author{Xiaoxu Xu\thanks{{\color{xxx}School of Mathematics and Statistics,} Xi'an Jiaotong University, Xi'an, Shaanxi, 710049, China ({\sf xuxiaoxu@xjtu.edu.cn})}
\and Guanghui Hu\thanks{{\color{hgh}Corresponding author: School of Mathematical Sciences and LPMC}, Nankai University, Tianjin, 300071, China ({\sf ghhu@nankai.edu.cn}).}}

\date{}
\end{titlepage}
\maketitle
\vspace{.2in}
\begin{abstract}
This paper is concerned with an inverse obstacle problem for the Laplace's equation.
The aim is to recover the constant conductivity coefficient in the equation and the boundary of a Dirichlet polygonal obstacle from a single pair of Cauchy data.
Uniqueness results are established under some a priori assumptions on the input boundary value data.
A domain-defined sampling method, based on the factorization method originating from inverse acoustic scattering, has been proposed to recover both the constant conductivity coefficient and the polygonal obstacle.
A hybrid strategy, which combines the sampling method and iterative scheme, is employed {\color{hgh}to reconstruct} the location and shape of the obstacle.
Numerical examples indicate that our method is efficient.

\vspace{.2in} {\bf Keywords}: inverse problem, uniqueness, single pair of Cauchy data, coefficient recovery, polygonal obstacle, factorization method.
\end{abstract}

\section{Introduction}

Suppose that $D$ is a convex polygon contained in the interior of a disk $B$.
We consider the following {\color{hgh} elliptic} boundary value problem:
\be\label{1}
{\rm div}(\gamma\nabla u)=0 & \text{in}\;B\ba\ov{D},\\ \label{2}
u=f & \text{on}\;\pa B,\\ \label{3}
u=0 & \text{on}\;\pa D.
\en
Here, $f$ is called the Dirichlet boundary value of $u$ on $\pa B$ and $D$ is assumed to be an obstacle with the homogeneous Dirichlet boundary condition $u=0$ on $\pa D$. {\color{hgh}In this paper}
the conductivity coefficient $\gamma>0$ in (\ref{1}) is assumed to be a constant in $B\ba\ov{D}$, which implies that $u$ is harmonic and analytic in $B\ba\ov{D}$ (see \cite[Definition 6.1 and Theorem 6.6]{Kress14}).
The existence of a unique solution to (\ref{1})--(\ref{3}) (including the case when (\ref{3}) is replaced by $u\!=\!g$ for some Dirichlet boundary value $g$ on $\pa D$) has been established by {\color{xxx}the} variational method (see \cite[Example 3.14]{Monk}).
Moreover, the {\color{xxx}well-posedness} of (\ref{1})--(\ref{3}) can also be established by the integral equation method {\color{xxx}(see Section \ref{s3.5} below,} see also \cite[Section 6.5]{Kress14} for boundary value problems in domains with corners).
Given a proper Dirichlet boundary value $f$, the Cauchy data $(u|_{\pa B},\gamma\pa_\nu u|_{\pa B})$ are uniquely determined since the boundary value problem (\ref{1})--(\ref{3}) is well-posed.
Here, $\nu$ denotes the unit normal vector to $\pa B$ directed into the exterior of $B$.
The inverse problem we consider in this paper is to determine $\gamma$ and $D$ from a single pair of Cauchy data $(u|_{\pa B},\gamma\pa_\nu u|_{\pa B})$.
Both uniqueness results and numerical methods of this inverse problem will be considered in the sequel.

Our studies are close to the existing references
\cite{Ike1998,Ike2002,WHB21,WHB23} where the enclosure method, range test and no-response test were applied to inverse conductivity problems with a polygonal interface (see also \cite{PS2010} for inverse electromagnetic scattering from polyhedral scatterers).
The idea of this paper is motivated by the one-wave factorization method proposed in \cite{EH2019,HL2020,MH22,MH2022} for inverse {\color{hgh} time-harmonic} scattering problems.
While the classical factorization method by Kirsch \cite{Kirsch1998} makes use of the scattering data over all incident and observation directions, the one-wave factorization method only requires the far-field data of a single incident plane wave for reconstructing scatterers of polygonal/polyhedral type.
It is particularly interesting within this paper that the constant conductivity coefficient can be also recovered by the one-wave factorization method.
All of the above-mentioned methods belong to the class of domain-defined sampling methods and are closely related to the analytical continuation of the solution (see \cite[Chapter 15]{NP}).
After finding a rough shape of the obstacle by the one-wave factorization method, we then employ a Newton-type iterative scheme to get a more precise reconstruction of the shape, which however relies heavily on {\color{hgh}proper} initial {\color{xxx}guesses} and {\color{hgh} efficient} forward solvers.

The remaining part of this paper is organized as follows.
Some preliminary results from factorization method based on Dirichlet-to-Neumann operators will be given in Section \ref{s2}.
Section \ref{s3} is devoted to the uniqueness results and numerical methods of the inverse problem from a single pair of Cauchy data.
Details about numerical simulation of the forward problem will be described in Section \ref{s3.5}.
Section \ref{sec_Newton} is concerned with the numerical implementation of Newton's iteration method.
Numerical examples will be reported in Section \ref{s4}.
Finally, a conclusion will be described in Section \ref{s5}.

\section{Preliminary results}\label{s2}
\setcounter{equation}{0}

In this section, we will introduce the factorization method for an elliptic boundary value problem \cite[Chapter 6]{Kirsch08}.
To circumvent the use of modified Sobolev spaces (see \cite[(6.6)--(6.8)]{Kirsch08}), we will develop the factorization method based on Dirichlet-to-Neumann operators.

We begin with the following single- and double-layer potentials with density $\varphi$:
\ben
(\mathcal S_{k,\pa\Om}\varphi)(x)=\int_{\pa\Om}\Phi_k(x,y)\varphi(y)ds(y),\quad x\in\R^2\ba\pa\Om,\\
(\mathcal D_{k,\pa\Om}\varphi)(x)=\int_{\pa\Om}\frac{\pa\Phi_k(x,y)}{\pa\nu(y)}\varphi(y)ds(y),\quad x\in\R^2\ba\pa\Om,
\enn
where $\Om$ is a domain in $\R^2$ with boundary $\pa\Om$, $\Phi_k(x,y)$ is the fundamental solution to the equation $\Delta u+k^2u=0$ (see \cite[Theorem 6.2]{Kress14} for $k=0$ and \cite[(3.106)]{CK19} for $k\neq0$), i.e.,
\ben
\Delta_x\Phi_k(x,y)+k^2\Phi_k(x,y)={\color{xxx}-\delta(x-y).}
\enn
Define the following boundary integral operators with density $\varphi$:
\ben
(S_{k,\pa\Om\ra\pa\widetilde\Om}\varphi)(x)=2\int_{\pa\Om}\Phi_k(x,y)\varphi(y)ds(y),\quad x\in\pa\widetilde\Om,\\
(K_{k,\pa\Om\ra\pa\widetilde\Om}\varphi)(x)=2\int_{\pa\Om}\frac{\pa\Phi_k(x,y)}{\pa\nu(y)}\varphi(y)ds(y),\quad x\in\pa\widetilde\Om,\\
(K'_{k,\pa\Om\ra\pa\widetilde\Om}\varphi)(x)=2\frac{\pa}{\pa\nu(x)}\int_{\pa\Om}\Phi_k(x,y)\varphi(y)ds(y),\quad x\in\pa\widetilde\Om,\\
(T_{k,\pa\Om\ra\pa\widetilde\Om}\varphi)(x)=2\frac{\pa}{\pa\nu(x)}\int_{\pa\Om}\frac{\pa\Phi_k(x,y)}{\pa\nu(y)}\varphi(y)ds(y),\quad x\in\pa\widetilde\Om,
\enn
where $\widetilde\Om$ is also a domain in $\R^2$ with boundary $\pa\widetilde\Om$.
For the case when $\pa\Om=\pa\widetilde\Om$, we define $S_{k,\pa\Om}:=S_{k,\pa\Om\ra\pa\Om}$, $K_{k,\pa\Om}:=K_{k,\pa\Om\ra\pa\Om}$, $K'_{k,\pa\Om}:=K'_{k,\pa\Om\ra\pa\Om}$ and $T_{k,\pa\Om}:=T_{k,\pa\Om\ra\pa\Om}$.
In the sequel, we will use the above notations with $\pa\Om\in\{\pa B,\pa D\}$, $\pa\widetilde\Om\in\{\pa B,\pa D\}$, and $k\in\{0,i\}$.
Here, $B$ and $D$ are given as in (\ref{1})--(\ref{3}).

\subsection{Boundary value problems and Dirichlet-to-Neumann operators}

Consider the Dirichlet boundary value problem
\be\label{5}
\Delta u_0=0 & \text{in}\; B,\\ \label{6}
u_0=f & \text{on}\;\pa B.
\en
The {\color{xxx} well-posedness of (\ref{5})--(\ref{6}) for $f\in H^{1/2}(\pa B)$ can be established by the} variational method (see \cite[Theorem 3.14]{Monk}).
However, {\color{hgh}we still need to investigate the {\color{xxx}well-posedness} of (\ref{5})--(\ref{6}) for $f\in H^{-1/2}(\pa B)$ (e.g., \cite[Section 10.2]{CK19}) in order to
establish the factorization method based on {\color{xxx}Dirichlet-to-Neumann operators}.}
To this end, we introduce several Sobolev spaces.
Let $\Om$ be a bounded domain with {\color{xxx}$C^2$ boundary} $\pa\Om$.
Define
\ben
H_\Delta^1(\Om):=\{u\in H^1(\Om):\Delta u\in L^2(\Om)\},\\
L_\Delta^2(\Om):=\{u\in L^2(\Om):\Delta u\in L^2(\Om)\},
\enn
where the Laplacian $\Delta$ is understood in the distributional sense.
Define the norms
\ben
\|u\|_{H_\Delta^1(\Om)}^2=\|u\|_{H^1(\Om)}^2+\|\Delta u\|_{L^2(\Om)}^2,\quad\forall u\in H_\Delta^1(\Om),\\
\|u\|_{L_\Delta^2(\Om)}^2=\|u\|_{L^2(\Om)}^2+\|\Delta u\|_{L^2(\Om)}^2,\quad\forall u\in L_\Delta^2(\Om).
\enn
Then we have the following trace theorem for $H_\Delta^1(\Om)$ and $L_\Delta^2(\Om)$.
\begin{theorem}\label{0404-3}
There exist constants $C_1,C_2,C_3>0$ such that
\be\label{0405-1}
\|\pa_\nu u\|_{H^{-1/2}(\pa \Om)}\leq C_1\|u\|_{H_\Delta^1(\Om)},\quad\forall u\in H_\Delta^1(\Om),\\ \label{0404-1}
\|u\|_{H^{-1/2}(\pa \Om)}\leq C_2\|u\|_{L_\Delta^2(\Om)},\quad\forall u\in L_\Delta^2(\Om),\\ \label{0404-2}
\|\pa_\nu u\|_{H^{-3/2}(\pa \Om)}\leq C_3\|u\|_{L_\Delta^2(\Om)},\quad\forall u\in L_\Delta^2(\Om).
\en
Moreover, for any $f\in H^{-1/2}(\pa \Om)$ the following statements are true:

(a) There exists {\color{xxx}$u_f\in H^1_\Delta(\Om)$} such that {\color{xxx}$\pa_\nu u_f=f$} in $H^{-1/2}(\pa \Om)$ and
\ben
{\color{xxx}\|u_f\|_{H_\Delta^1(\Om)}}\leq C_4\|f\|_{H^{-1/2}(\pa \Om)},
\enn
where $C_4>0$ is a constant independent of $f$;

(b) There exists $v_f\in L_\Delta^2(\Om)$ such that $v_f=f$ in $H^{-1/2}(\pa \Om)$ and
\ben
\|v_f\|_{L_\Delta^2(\Om)}\leq C_5\|f\|_{H^{-1/2}(\pa \Om)},
\enn
where $C_5>0$ is a constant independent of $f$.
\end{theorem}
\begin{proof}
For the proof of (\ref{0405-1}) we refer the reader to \cite[Page 53]{CK19}, and for the proofs of (\ref{0404-1}) and (\ref{0404-2}) we refer the reader to \cite[Section 10.2]{CK19}.
Statement (a) follows directly from the existence of a unique solution $u_f\in H^1(\Om)$ to the Neumann boundary value problem (\cite[Theorem 3.15]{Monk})
\ben
\Delta u_f-u_f=0 && \text{in }\Om,\\
\pa_\nu u_f=f && \text{on }\Om.
\enn
It remains to prove statement (b).
Since $C^{0,\alpha}(\pa \Om)$ is dense in $H^{-1/2}(\pa \Om)$, there exists a sequence $\{f_j\}\subset C^{0,\alpha}(\pa \Om)$ such that $\|f_j-f\|_{H^{-1/2}(\pa \Om)}\ra0$ as $j\ra\infty$.
Note that $S_{i,\pa \Om}:H^{-3/2}(\pa \Om)\ra H^{-1/2}(\pa \Om)$ has a bounded inverse (see \cite[Page 391]{CK19}).
By the jump relations \cite[Theorem 3.1]{CK19}, it holds that $\mathcal S_{i,\pa \Om}(S_{i,\pa \Om}^{-1}f_j)=f_j$ in $C^{0,\alpha}(\pa \Om)$ and thus $\mathcal S_{i,\pa \Om}(S_{i,\pa \Om}^{-1}f_j)=f_j$ in $H^{-1/2}(\pa \Om)$.
Passing to the limit $j\ra\infty$, we conclude from \cite[Theorem 10.12]{CK19} that $v_f:=\mathcal S_{i,\pa \Om}(S_{i,\pa \Om}^{-1}f)\in L_\Delta^2(\Om)$ satisfies $v_f=f$ in $H^{-1/2}(\pa \Om)$ and $\|v_f\|_{L^2_\Delta(\Om)}\leq C_5\|f\|_{H^{-1/2}(\pa \Om)}$.
\end{proof}
By Theorem \ref{0404-3}, we can establish the {\color{xxx} well-posedness} of (\ref{5})--(\ref{6}) for $f\in H^{-1/2}(\pa B)$.
\begin{theorem}\label{0404-4}
(i) For any $f\in H^{1/2}(\pa B)$, the boundary value problem (\ref{5})--(\ref{6}) has a unique solution $u_0\in H^1_\Delta(B)$ satisfying
\ben
\|u_0\|_{H^1_\Delta(B)}\leq C\|f\|_{H^{1/2}(\pa B)},
\enn
where $C>0$ is a constant independent of $f$.

(ii) For any $f\in H^{-1/2}(\pa B)$, the boundary value problem (\ref{5})--(\ref{6}) has a unique solution $u_0\in L^2_\Delta(B)$ satisfying
\ben
\|u_0\|_{L^2_\Delta(B)}\leq C\|f\|_{H^{-1/2}(\pa B)},
\enn
where $C>0$ is a constant independent of $f$.
\end{theorem}
\begin{proof} Statement (i) follows directly from \cite[Example 5.15]{Cakoni14}.
It suffices to prove statement (ii).
By Theorem \ref{0404-3} (b), there exists $v_f\in L_\Delta^2(B)$ such that $v_f=f$ in $H^{-1/2}(\pa B)$ and $\|v_f\|_{L_\Delta^2(B)}\leq C_5\|f\|_{H^{-1/2}(\pa B)}$.
Then $w:=u_0-v_f$ satisfies the following boundary problem
\ben
\Delta w=-\Delta v_f && \text{in}\; B,\\
w=0 && \text{on}\;\pa B.
\enn
According to \cite[Example 5.15]{Cakoni14}, there exists a constant $\widetilde C>0$ such that
\ben
\|w\|_{H^1(B)}\leq\widetilde C\|\Delta v_f\|_{L^2(B)}.
\enn
Therefore,
\ben
\|u_0\|_{L^2_\Delta(B)}\leq\|w\|_{L^2_\Delta(B)}+\|v_f\|_{L^2_\Delta(B)}\leq\sqrt{\widetilde C^2+1}\|\Delta v_f\|_{L^2(D)}+\|v_f\|_{L^2_\Delta(B)}\leq C\|f\|_{H^{-1/2}(\pa B)},
\enn
where $C=C_5(\sqrt{\widetilde C^2+1}+1)>0$ is a constant independent of $f$.
The proof is now complete.
\end{proof}

Next, consider the following boundary value problem
\be\label{1'}
\Delta u=0 && \text{in }B\ba\ov{D},\\ \label{2'}
u=f && \text{on }\pa B,\\ \label{3'}
u=0 && \text{on }\pa D.
\en
The proof of Theorem \ref{0404-4} cannot be directly employed for the well-posedness of (\ref{1'})--(\ref{3'}) since $D$ is a polygon but Theorem \ref{0404-3} is valid for domain $\Om$ of class $C^2$ (see \cite[Section 10.2]{CK19}).
\begin{theorem}\label{0405-2}
(i) For any $f\in H^{1/2}(\pa B)$, the boundary value problem (\ref{1'})--(\ref{3'}) has a unique solution $u\in H^1_\Delta(B\ba\ov{D})$ satisfying
\ben
\|u\|_{H^1_\Delta(B\ba\ov{D})}\leq C\|f\|_{H^{1/2}(\pa B)},
\enn
where $C>0$ is a constant independent of $f$.

(ii) Assume $B_0$ is an open subset of $B$ such that $\ov{D}\subset B_0\subset\ov{B_0}\subset B$.
For any $f\in H^{-1/2}(\pa B)$, the boundary value problem (\ref{1'})--(\ref{3'}) has a unique solution $u\in L^2_\Delta(B\ba\ov{D})$ satisfying
\ben
\|u\|_{H_\Delta^1(B_0\ba\ov{D})} +\|u\|_{L^2_\Delta(B\ba\ov{B_0})}\leq C\|f\|_{H^{-1/2}(\pa B)},
\enn
where $C>0$ is a constant independent of $f$.
\end{theorem}
\begin{proof}
Statement (i) follows directly from \cite[Theorem 3.14]{Monk}.
It suffices to prove statement (ii).
Let $\chi\in C^\infty(\ov{B})$ be a cut-off function such that $\chi(x)=0$ for $x\in B_0$ and $\chi(x)=1$ for $x$ in the vicinity of $\pa B$.
By Theorem \ref{0404-3} (b), there exists $v_f\in L_\Delta^2(B)$ such that $v_f=f$ in $H^{-1/2}(\pa B)$ and
\be\label{0503-1}
\|v_f\|_{L_\Delta^2(B)}\leq C_5\|f\|_{H^{-1/2}(\pa B)}.
\en
Set $w:=u-\chi v_f$, then $w=u$ in $B_0\ba\ov{D}$.
Moreover, $w$ satisfies the boundary value problem
\ben
\Delta w=-\Delta(\chi v_f) && \text{in }B\ba\ov{D},\\
w=0 && \text{on }\pa B,\\
w=0 && \text{on }\pa D.
\enn
Noting that $\Delta(\chi v_f)=v_f\Delta\chi+2\nabla\chi\cdot\nabla v_f+\chi\Delta v_f\in H^{-1}(B\ba\ov{D})$, we conclude from \cite[Theorem 3.14]{Monk} that
\ben
\|w\|_{H^1(B\ba\ov{D})}\leq\widetilde C\|\Delta(\chi v_f)\|_{H^{-1}(B\ba\ov{D})}\leq\widehat C\|v_f\|_{L_\Delta^2(B\ba\ov{D})},
\enn
where $\widetilde C,\widehat C>0$ are constants independent of $f$.
Therefore,
\ben
\|u\|_{H_\Delta^1(B_0\ba\ov{D})}=\|{\color{xxx}u}\|_{H^1(B_0\ba\ov{D})}=\|w\|_{H^1(B_0\ba\ov{D})}\leq\|w\|_{H^1(B\ba\ov{D})}\leq\widehat C\|v_f\|_{L_\Delta^2(B\ba\ov{D})},\\
\|u\|_{L^2_\Delta(B\ba\ov{B_0})}=\|u\|_{L^2(B\ba\ov{B_0})}\leq\|w\|_{L^2(B\ba\ov{B_0})}+\|\chi v_f\|_{L^2(B\ba\ov{B_0})}\leq\widehat C\|v_f\|_{L_\Delta^2(B\ba\ov{D})}.
\enn
In view of \eqref{0503-1}, the proof is complete.
\end{proof}
\begin{remark}\label{0413-1}
(i) {\color{hgh} $u_0$ and $u$ are harmonic and analytic in $B$ and $B\ba\ov{D}$, respectively;} see \cite[Definition 6.1 and Theorem 6.6]{Kress14}.

(ii) {\color{xxx}Let $B_0$, $\chi$ and $v_f$ be given as in the proof of Theorem \ref{0405-2} (ii).}
For any $f\in H^{-1/2}(\pa B)$ and $g\in H^{1/2}(\pa D)$, the boundary value problem
\be\label{0503-2}
\Delta u=0 && \text{in }B\ba\ov{D},\\ \label{0503-3}
u=f && \text{on }\pa B,\\ \label{0503-4}
u=g && \text{on }\pa D,
\en
has a unique solution {\color{xxx}$u\in L^2_\Delta(B\ba\ov{D})$} satisfying
\ben
\|u\|_{H_\Delta^1(B_0\ba\ov{D})} +\|u\|_{L^2_\Delta(B\ba\ov{B_0})}\leq C(\|f\|_{H^{-1/2}(\pa B)}+\|g\|_{H^{1/2}(\pa D)}),
\enn
where $C>0$ is a constant independent of $f$ and $g$.
{\color{xxx}Actually, by} \cite[Theorem 3.12]{Monk} there exists $u_g\in H^1_\Delta(B\ba\ov{D})$ such that $u_g=g$ in $H^{1/2}(\pa B)$ and $\|u_g\|_{H_\Delta^1(\Om)}\leq c\|g\|_{H^{1/2}(\pa \Om)}$ for a constant $c>0$ independent of $g$.
Therefore, the assertion follows by setting $w:=u-[(1-\chi)u_g+\chi v_f]$ and a similar argument as above.
Obviously, the solution $u$ to \eqref{0503-2}--\eqref{0503-4} is also harmonic and analytic in $B\ba\ov{D}$.

(iii) In view of Theorem \ref{0404-4}, define the Dirichlet-to-Neumann operator $\Lambda_0$ corresponding to (\ref{5})--(\ref{6}) by $\Lambda_0f:=\pa_\nu u_0$.
By Theorems \ref{0404-3} and \ref{0404-4}, we know $\Lambda_0:H^{1/2}(\pa B)\ra H^{-1/2}(\pa B)$ and $\Lambda_0:H^{-1/2}(\pa B)\ra H^{-3/2}(\pa B)$ are bounded.
By the interpolation property of the Sobolev
spaces (see \cite[Theorem 8.13]{Kress14}), $\Lambda_0:H^s(\pa B)\ra H^{s-1}(\pa B)$ is bounded for $s\in[-\frac12,\frac12]$.

(iv) In view of Theorem \ref{0405-2}, define the Dirichlet-to-Neumann operator $\Lambda_D$ corresponding to (\ref{1'})--(\ref{3'}) by $\Lambda_Df:=\pa_\nu u$.
By Theorems \ref{0404-3} and \ref{0405-2}, we know $\Lambda_D:H^{1/2}(\pa B)\ra H^{-1/2}(\pa B)$ and $\Lambda_D:H^{-1/2}(\pa B)\ra H^{-3/2}(\pa B)$ are bounded.
By the interpolation property of the Sobolev
spaces (see \cite[Theorem 8.13]{Kress14}), $\Lambda_D:H^s(\pa B)\ra H^{s-1}(\pa B)$ is bounded for $s\in[-\frac12,\frac12]$.
\end{remark}
Now we return to the boundary value problem (\ref{1})--(\ref{3}).
We define the Dirichlet-to-Neumann operator $\Lambda_{\ga,D}$ by $\Lambda_{\ga,D} f:=\gamma\pa_\nu u$ on $\pa B$, where $u$ solves (\ref{1})--(\ref{3}).
Obviously, (\ref{1'})--(\ref{3'}) is a special case of (\ref{1})--(\ref{3}) with $\ga=1$ and thus $\Lambda_{\ga,D}=\ga\Lambda_D$.
With this notation, the Cauchy data $(u|_{\pa B},\gamma\pa_\nu u|_{\pa B})$ to (\ref{1})--(\ref{3}) can be represented as $(f,\ga\Lambda_Df)$.

\subsection{Factorization method}

Below we shall derive a factorization of $\Lambda_D-\Lambda_0$.
To this end, we define the operator $G_D$ by $G_Dg=\pa_\nu w$, where $w$ is the unique solution to the following boundary value problem:
\be\label{7}
\Delta w=0 & \text{in}\;B\ba\ov{D},\\ \label{8}
w=0 & \text{on}\;\pa B,\\ \label{9}
w=g & \text{on}\;\pa D.
\en
\begin{theorem}\label{t3.3}
The operator $G_D:H^{1/2}(\pa D)\ra L^2(\pa B)$ is well-defined, bounded and compact.
\end{theorem}

\begin{proof}
Let $g\in H^{1/2}(\pa D)$ in \eqref{9}.
According to \cite[Theorem 3.14]{Monk}, (\ref{7})--(\ref{9}) has a unique solution $w\in H^1(B\ba\ov{D})$ satisfying
\be\label{10}
\left\|w\right\|_{H^1(B\ba\ov{D})}\leq C\left\|g\right\|_{H^{1/2}(\pa D)},
\en
where the constant $C>0$ is independent of $g$.
{\color{xxx}Since} $D$ is contained in the interior of $B$, for any $J\!\in\!\Z_+$ there exists a set of domains $\{\Om_j\}_{j=1}^J$ with boundaries $\pa\Om_j\in C^\infty${\color{xxx}, $j\!=\!1,\!\cdots\!,J$,} such that
\ben
D\subset\ov{D}\subset\Om_1\subset\ov{\Om_1}\subset\Om_2\subset\ov{\Om_2}\subset\cdots\subset\Om_J\subset\ov{\Om_J}\subset B.
\enn
We claim that for any $m\in\{1,2,\cdots,J\}$ there exists a constant $C_m>0$ independent of $g$ such that $\|w\|_{H^m(B\ba\ov{\Om_m})}\leq C_m\|g\|_{H^{1/2}(\pa D)}$.
Actually, by setting $\chi_1\in C^\infty(\ov{B})$ to be a cut-off function such that $\chi_1(x)=0$ for $x\in\Om_1$ and $\chi_1(x)=1$ for $x\in B\ba\ov{\Om_2}$, we know $\tilde w_1:=w\chi_1$ is the unique solution to
\ben
\Delta\tilde w_1=\tilde p_1\quad\text{in}\;B\ba\ov{\Om}_1,&&\tilde w_1=0\quad\text{on}\;\pa B\cup\pa\Om_1,
\enn
where $\tilde p_1:=w\Delta\chi_1+2\nabla w\cdot\nabla\chi_1$ satisfies $\|\tilde p_1\|_{L^2(B\ba\ov{\Om_1})}\leq\widetilde C_1\|g\|_{H^{1/2}(\pa B)}$ due to (\ref{10}).
Here $\widetilde C_1>0$ is a constant independent of $g$.
It follows from the regularity of elliptic equation (see \cite[Theorem 8.12]{GT}) that $\|w\|_{H^2(B\ba\ov{\Om_2})}\leq\|\tilde w_1\|_{H^2(B\ba\ov{\Om_1})}\leq C_2\|g\|_{H^{1/2}(\pa B)}$.

Now, suppose that $\|w\|_{H^m(B\ba\ov{\Om_m})}\leq C_m\|g\|_{H^{1/2}(\pa D)}$ for some $m\in\Z_+$.
Set $\chi_m\in C^\infty(\ov{B})$ to be a cut-off function such that $\chi_m(x)=0$ for $x\in\Om_m$ and $\chi_m(x)=1$ for $x\in B\ba\ov{\Om_{m+1}}$.
Then $\tilde w_m:=w\chi_m$ is the unique solution to
\ben
\Delta\tilde w_m=\tilde p_m\quad\text{in }B\ba\ov{\Om_m},&&\tilde w_m=0\quad\text{on }\pa B\cup\pa\Om_m,
\enn
where $\tilde p_m:=w\Delta\chi_m+2\nabla w\cdot\nabla\chi_m$ satisfies $\|\tilde p_m\|_{H^{m-1}(B\ba\ov{\Om_m})}\leq\widetilde C_m\|g\|_{H^{1/2}(\pa D)}$ due to the inductive hypothesis.
Here $\widetilde C_m>0$ is a constant independent of $g$.
It follows from the regularity of elliptic equation (see \cite[Theorem 8.13]{GT}) that $\|w\|_{H^{m+1}(B\ba\ov{\Om_{m+1}})}\leq\|\tilde w_m\|_{H^{m+1}(B\ba\ov{\Om_m})}\leq C_{m+1}\|g\|_{H^{1/2}(\pa D)}$.

By induction, we have $\|w\|_{H^m(B\ba\ov{\Om_m})}\leq C_m\|g\|_{H^{1/2}(\pa D)}$ for all $m\in\{1,2,\cdots,J\}$ and thus $\|\pa_\nu w\|_{H^{m-3/2}(\pa B)}\leq c_m\|g\|_{H^{1/2}(\pa D)}$ for all $m\in\{1,2,\cdots,J\}$.
Here $c_m>0$ is a constant independent of $g$ for all $m\in\{1,2,\cdots,J\}$.
Finally, the compactness of $G_D:H^{1/2}(\pa D)\ra L^2(\pa B)$ follows from the compact embedding of $H^{m-3/2}(\pa B)$ into $L^2(\pa B)$ provided $J\geq m\geq2$.
\end{proof}

To obtain a factorization of $\Lambda_D-\Lambda_0$, we need to introduce the boundary integral operators defined on $\pa D$.
We begin with the Green's function $K(x,y)$ to the Laplace's equation in $B$ with the Dirichlet boundary condition:
\ben
K(x,y):=\Phi_0(x,y)+k(x,y),\quad x,y
\in B,\quad x\neq y,
\enn
where $\Phi_0(x,y)$ is the fundamental solution to the Laplace's equation in $\R^2$ (see \cite[Theorem 6.2]{Kress14}) and $u_0\!=\!k(\cdot,y)$ solves (\ref{5})--(\ref{6}) with $f\!:=\!-\Phi_0(\cdot,y)$ on $\pa B$.
We note that $K(x,y)$ is harmonic and analytic in $x\!\in\!B\ba\{y\}$ and $K(\cdot,y)\!=\!0$ on $\pa B$.
Moreover, we have the reciprocity relation (symmetric property) for $K(x,y)$ as follows.
\begin{lemma}\label{lem211025}
$k(x,y)=k(y,x)$ and $K(x,y)=K(y,x)$ for all $x,y\in B$, $x\neq y$.
\end{lemma}
\begin{proof}
It suffices to prove $k(x,y)=k(y,x)$ for all $x,y\in B$.
Setting $B_R$ to be a disk centered at origin with radius $R>0$ large enough and using Green's second theorem (see \cite[(6.2)]{Kress14}) in $B_R\ba\ov{B}$, we have for $x,y\in B$ that
\be\no
&&\int_{\pa B}\left\{\frac{\pa\Phi_0(z,y)}{\pa\nu(z)}\Phi_0(z,x)-\Phi_0(z,y)\frac{\pa\Phi_0(z,x)}{\pa\nu(z)}\right\}ds(z)\\ \label{211026-2}
&=&\int_{\pa B_R}\left\{\frac{\pa\Phi_0(z,y)}{\pa\nu(z)}\Phi_0(z,x)-\Phi_0(z,y)\frac{\pa\Phi_0(z,x)}{\pa\nu(z)}\right\}ds(z),
\en
where $\nu$ denotes the unit normal vector to $\pa B$ or $\pa B_R$ directed into the exterior of $B$ or $B_R$, respectively.
Note that for $x,y\in B$ and $|z|$ large enough there exists a constant $C>0$ such that
\ben
&|\Phi_0(x,z)-\Phi_0(y,z)|\leq|x-y|\max_{x\in B}|\nabla_x\Phi_0(x,z)|\leq C/|z|,&\\
&|\pa_{\nu(z)}\Phi_0(x,z)-\pa_{\nu(z)}\Phi_0(y,z)|\leq|x-y|\max_{x\in B}\left|\nabla_x[\pa_{\nu(z)}\Phi_0(x,z)]\right|\leq C/|z|^2.&
\enn
Therefore,
\ben
&&\left|\int_{\pa B_R}\left\{\frac{\pa\Phi_0(z,y)}{\pa\nu(z)}[\Phi_0(z,x)-\Phi_0(z,y)]-\Phi_0(z,y)\left[\frac{\pa\Phi_0(z,x)}{\pa\nu(z)}-\frac{\pa\Phi_0(z,y)}{\pa\nu(z)}\right]\right\}ds(z)\right|\\
&\leq&2\pi R\left(\frac CR\cdot\frac CR+\ln R\cdot\frac C{R^2}\right).
\enn
Passing to the limit $R\!\ra\!\infty$, we deduce from (\ref{211026-2}) that
\ben
\int_{\pa B}\left\{\frac{\pa\Phi_0(z,y)}{\pa\nu(z)}\Phi_0(z,x)-\Phi_0(z,y)\frac{\pa\Phi_0(z,x)}{\pa\nu(z)}\right\}ds(z)=0.
\enn
Applying Green's second theorem over $B$ to $k(\cdot,y)$ and $k(\cdot,x)$, we have for $x,y\!\in\!B$ that
\ben
\int_{\pa B}\left\{\frac{\pa k(z,y)}{\pa\nu(z)}k(z,x)-k(z,y)\frac{\pa k(z,x)}{\pa\nu(z)}\right\}ds(z)=0.
\enn
Using the Green's formula (see \cite[Theorem 6.5]{Kress14}) we have for $x,y\!\in\!B$ that
\ben
k(x,y)&=&\int_{\pa B}\left\{\frac{\pa k(z,y)}{\pa\nu(z)}\Phi_0(z,x)-k(z,y)\frac{\pa\Phi_0(z,x)}{\pa\nu(z)}\right\}ds(z),\\
-k(y,x)&=&-\int_{\pa B}\left\{\frac{\pa k(z,x)}{\pa\nu(z)}\Phi_0(z,y)-k(z,x)\frac{\pa\Phi_0(z,y)}{\pa\nu(z)}\right\}ds(z).
\enn
Taking the sum of the above four equalities gives
\ben
k(x,y)-k(y,x)=\int_{\pa B}\left\{\frac{\pa K(z,y)}{\pa\nu(z)}K(z,x)-K(z,y)\frac{\pa K(z,x)}{\pa\nu(z)}\right\}ds(z),\quad x,y\in B.
\enn
Now the proof is completed by using the boundary condition $K(\cdot,x)\!=\!K(\cdot,y)\!=\!0$ on $\pa B$.
\end{proof}
Using the Green's function {\color{xxx}$K(x,y)$} one may represent the solution $u_0$ to (\ref{5})--(\ref{6}) in terms of the boundary data $f$.
\begin{lemma}\label{t3.4}
For $f\in H^s(\pa B)$ with $s\in[-\frac12,\frac12]$, the solution $u_0\in L^2_\Delta(B)$ to (\ref{5})--(\ref{6}) is given by
\be\label{211025-1}
u_0(x)=-\int_{\pa B}\frac{\pa K(x,y)}{\pa\nu(y)}f(y)ds(y),\quad x\in B.
\en
In particular, $u_0\in H^1_\Delta(B)$ provided $f\in H^{1/2}(\pa B)$.
\end{lemma}

\begin{proof}
First, we assume that $f\in C^2(\ov{B})$.
For any fixed $x\in B$, we observe that $u_0$ and $k(\cdot,x)$ are harmonic in $B$ and $k(x,\cdot)=k(\cdot,x)$ (see Lemma \ref{lem211025}).
Applying Green's second theorem (see \cite[(6.2)]{Kress14}) over $B$ to $u_0$ and $k(x,\cdot)$ {\color{xxx}shows that}
\ben
0=\int_{\pa B}\left\{\frac{\pa u_0}{\pa\nu}(y)k(x,y)-\frac{\pa k(x,y)}{\pa\nu(y)}u_0(y)\right\}ds(y),\quad x\in B.
\enn
Using the Green's formula (see \cite[Theorem 6.5]{Kress14}) we have
\ben
u_0(x)=\int_{\pa B}\left\{\frac{\pa u_0}{\pa\nu}(y)\Phi_0(x,y)-\frac{\pa\Phi_0(x,y)}{\pa\nu(y)}u_0(y)\right\}ds(y),\quad x\in B.
\enn
Since $u_0=f$ on $\pa B$ and $K(x,\cdot)=K(\cdot,x)=0$ on $\pa B$, the representation (\ref{211025-1}) follows by taking the sum of the above two equalities.

Finally, in view of Theorem \ref{0404-4}, Remark \ref{0413-1} {\color{xxx}(iii)} and \cite[Theorem 10.12]{CK19}, {\color{xxx}we can obtain} the results for $f\in H^s(\pa B)$ with $s\in[-\frac12,\frac12]$ by denseness arguments.
\end{proof}

For $f\in L^2(\pa B)$ and $\varphi\in H^{-1/2}(\pa D)$, define the following integral operators
\ben
(Hf)(x):=-\int_{\pa B}\frac{\pa K(x,y)}{\pa\nu(y)}f(y)ds(y),\quad x\in\pa D,\\
(S_{\pa D}\varphi)(x):=\int_{\pa D}K(x,y)\varphi(y)ds(y),\quad x\in\pa D.
\enn
The properties of the above integral operators are given in the following theorem.
\begin{theorem}\label{t+}
(i) $H:L^2(\pa B)\ra H^{1/2}(\pa D)$ is bounded, compact and injective.

(ii) $S_{\pa D}:H^{-1/2}(\pa D)\ra H^{1/2}(\pa D)$ is bounded, self-adjoint and coercive, i.e., there exists a constant $c\!>\!0$ independent of $\varphi$ such that
\be\label{0504}
\langle S_{\pa D}\varphi,\varphi\rangle\geq c\left\|\varphi\right\|^2_{H^{-1/2}(\pa D)}\quad\text{for all}\;\varphi\in H^{-1/2}(\pa D),
\en
where $\langle\cdot,\cdot\rangle$ represents the sesquilinear duality pairing $\langle H^{1/2}(\pa D),H^{-1/2}(\pa D)\rangle$.
\end{theorem}

\begin{proof}
(i). Assume $f\in L^2(\pa B)$.
In view of Lemma \ref{t3.4}, $Hf$ is the trace of $u_0$ on $\pa D$ (see \cite[Theorem 3.38]{WM}).
Since $u_0$ defined by (\ref{211025-1}) is harmonic and analytic {\color{xxx}in $B$}, we conclude that $H:L^2(\pa B)\ra H^{1/2}(\pa D)$ is bounded and compact.
Now assume $Hf=0$ on $\pa D$, then $u_0$ defined by (\ref{211025-1}) satisfies
\ben
\Delta u_0=0 && \text{in }D,\\
u_0=0 && \text{on }\pa D.
\enn
The uniqueness of the above boundary value problem implies $u_0=0$ in $D$.
Hence, $u_0=0$ in $B$ by analyticity.
Then $f=0$ on $\pa B$ follows from Lemma \ref{t3.4}.
This shows the injectivity of $H$.

(ii). Noting that $K(x,y)$ has the same type of singularity for $x=y$ as $\Phi_0(x,y)$ and $\Phi_i(x,y)$, we deduce from the jump relations and regularity properties of boundary integral operators (see \cite{CK83,CK19} for domains of class $C^2$ and \cite{Costabel,WM} for Lipschitz domains) that $S_{\pa D}:H^{-1/2}(\pa D)\ra H^{1/2}(\pa D)$ is bounded for any $\varphi\in H^{-1/2}(\pa D)$ and the function $w$ defined by
\be\label{14}
w(x):=\int_{\pa D}K(x,y)\varphi(y)ds(y),\quad x\in\R^2\ba\pa D,
\en
satisfies the following jump relations:
\be\label{211025-2}
w_\pm\!=\!S_{\pa D}\varphi,\quad\pa_\nu w_+\!-\!\pa_\nu w_-\!=\!-\varphi&& \text{on}\;\pa D,
\en
where the subindex $+\, (-)$ indicates the limit as $x$ approaches $\pa D$ from outside (inside) of $D$, respectively.

It is easy to deduce from Lemma \ref{lem211025} that $S_{\pa D}:H^{-1/2}(\pa D)\ra H^{1/2}(\pa D)$ is self-adjoint.

To show that $S_{\pa D}:H^{-1/2}(\pa D)\ra H^{1/2}(\pa D)$ is injective, we assume $S_{\pa D}\varphi\!=\!0$ on $\pa D$.
Then the function $w$ defined by (\ref{14}) satisfies
\ben
\begin{cases}
\Delta w=0 & \text{in}\;B\ba\ov{D},\\
w=0 & \text{on}\;\pa D\cup\pa B,
\end{cases}\quad\text{and}\quad\begin{cases}
\Delta w=0 & \text{in}\;D,\\
w=0 & \text{on}\;\pa D.
\end{cases}
\enn
It follows from the uniqueness of above boundary value problems that $w=0$ in $D$ and thus, by (\ref{211025-2}), we get $\varphi=\pa_\nu w_--\pa_\nu w_+=0$ on $\pa D$.

We claim that $S_{\pa D}:H^{-1/2}(\pa D)\ra H^{1/2}(\pa D)$ has a bounded inverse.
Note that the operator $S_{\pa D}-S_{i,\pa D}:H^{-1/2}(\pa D)\ra H^{1/2}(\pa D)$ is compact {\color{xxx}due to the increased smoothness of
the integral kernel as compared with that of $S_{i,\pa D}$}.
It follows that $S_{\pa D}=S_{i,\pa D}+(S_{\pa D}-S_{i,\pa D})$ is a Fredholm operator of index zero since the inverse $S_{i,\pa D}^{-1}:H^{1/2}(\pa D)\ra H^{-1/2}(\pa D)$ is bounded (see \cite[Theorem 5.44]{KH}).
We also refer the reader to \cite[Theorem 7.6]{WM} for another proof that an integral operator, whose integral kernel has the same type of singularity for $x=y$ as $K(x,y)$, is a Fredholm operator of index zero.
Recall that $S_{\pa D}$ is injective.
By the Riesz-Fredholm theory we know that the inverse $S_{\pa D}^{-1}:H^{1/2}(\pa D)\ra H^{-1/2}(\pa D)$ is bounded.

We are now ready to prove the coercivity of $S_{\pa D}$.
Let $w$ be defined by (\ref{14}).
By the jump relations (\ref{211025-2}) and using Green's first theorem (see \cite[Theorem 6.3]{Kress14}) we have
\ben
\langle S_{\pa D}\varphi,\varphi\rangle&=&\int_{\pa D}w(\pa_\nu w_--\pa_\nu w_+)ds\\
&=&\int_D\left\{w\Delta w+|\nabla w|^2\right\}dx+\int_{B\ba\ov{D}}\left\{w\Delta w+|\nabla w|^2\right\}dx-\int_{\pa B}w\pa_\nu wds\\
&=&\int_D|\nabla w|^2dx+\int_{B\ba\ov{D}}|\nabla w|^2dx,
\enn
where we have used the fact that $\Delta w=0$ in $B\ba\pa D$ and $w=0$ on $\pa B$.
By the Poincar\'e inequality (see \cite[Lemma 3.13]{Monk}) and the trace theorem (see \cite[Theorem 3.9]{Monk}), we get
\ben
\langle S_{\pa D}\varphi,\varphi\rangle\geq\left\|\nabla w\right\|_{L^2(B\ba\ov{D})}^2\geq c_1\left\|w\right\|_{H^1(B\ba\ov{D})}^2\geq c_2\left\|S_{\pa D}\varphi\right\|_{H^{1/2}(\pa D)}^2
\enn
for some constants $c_1,c_2\!>\!0$ independent of $\varphi$ and $w$.
Noting that $S_{\pa D}^{-1}:H^{1/2}(\pa D)\ra H^{-1/2}(\pa D)$ is bounded, we finally arrive at \eqref{0504}.
\end{proof}

We are now in a position to derive a factorization of $\Lambda_D-\Lambda_0$.

\begin{theorem}\label{t3.5}
The following relation between $\Lambda_D-\Lambda_0$, $G_D$ and $S_{\pa D}$ holds:
\be\label{fac}
\Lambda_D-\Lambda_0=G_DS_{\pa D}^*G_D^*,
\en
where $G_D^*:L^2(\pa B)\ra H^{-1/2}(\pa D)$ and $S_{\pa D}^*:H^{-1/2}(\pa D)\ra H^{1/2}(\pa D)$ are the adjoint operators of $G_D$ and $S_{\pa D}$, respectively.
\end{theorem}

\begin{proof}
Given $f\in L^2(\pa B)$, we have $\Lambda_Df-\Lambda_0f=\pa_\nu u-\pa_\nu u_0\in H^{-1}(\pa B)$, where $u\in L_\Delta^2(B\ba\ov{D})$ solves (\ref{1})--(\ref{3}) and $u_0\in L_\Delta^2(B)$ solves (\ref{5})--(\ref{6}), respectively.
Noting that $w=u-u_0$ solves (\ref{7})--(\ref{9}) with $g=-u_0$ on $\partial D$ and using {\color{xxx}Remark} \ref{0413-1} (i) and Lemma \ref{t3.4}, we obtain
\be\label{12}
(\Lambda_D-\Lambda_0)f=G_D(\left.-u_0\right|_{\pa D})=-G_DHf.
\en
The $L^2$ adjoint $H^*:H^{-1/2}(\pa D)\ra L^2(\pa B)$ is given by
\ben
(H^*\varphi)(x):=-\int_{\pa D}\frac{\pa K(x,y)}{\pa\nu(x)}\varphi(y)ds(y),\quad x\in\pa B.
\enn
We observe that $H^*\varphi=\pa_\nu v$ on $\pa B$, where $v$ is defined by
\ben
v(x)=-\int_{\pa D}K(x,y)\varphi(y)ds(y),\quad x\in B\ba\ov{D}.
\enn
Since $v$ solves (\ref{7})--(\ref{9}) with $g\!=\!-S_{\pa D}\varphi$,
we have $H^*\!=\!-G_DS_{\pa D}$  and consequently
\be\label{13}
H=-S_{\pa D}^*G_D^*.
\en
Now, the factorization form \eqref{fac}  follows by combining (\ref{12}) and (\ref{13}).
\end{proof}

\begin{remark}\label{0915}
(i) It follows from Remark \ref{0413-1} (iii) and (iv) that $\Lambda_D-\Lambda_0$ is bounded from $H^s(\pa B)$ to $H^{s-1}(\pa B)$ for any $s\in[-\frac12,\frac12]$.
Further, since $w=u-u_0$ solves (\ref{7})--(\ref{9}) with $g=-u_0$ on $\partial D$, we conclude from the proof of Theorem \ref{t3.3} that $\Lambda_D-\Lambda_0:H^s(\pa B)\ra H^{m-3/2}(\pa B)$ is bounded for any $s\in[-\frac12,\frac12]$ and $m\in\Z_+$.
In particular, $\Lambda_D-\Lambda_0:L^2(\pa B)\ra L^2(\pa B)$ is bounded.

(ii) Noting that $H:L^2(\pa B)\ra H^{1/2}(\pa D)$ is injective, we conclude from (\ref{13}) that $G_D^*:L^2(\pa B)\ra H^{-1/2}(\pa D)$ is also injective and thus $G_D:H^{1/2}(\pa D)\ra L^2(\pa B)$ has a dense range.
\end{remark}

{\color{xxx}Since $G_D:H^{1/2}(\pa D)\ra L^2(\pa B)$ is compact} (see Theorem \ref{t3.3}) and $S_{\pa D}:H^{-1/2}(\pa D)\ra H^{1/2}(\pa D)$ is self-adjoint (see Theorem \ref{t+}),
it follows from \eqref{fac} that $\Lambda_D-\Lambda_0:L^2(\pa B)\ra L^2(\pa B)$ is {\color{xxx}compact and} self-adjoint.
Combining Theorems \ref{t+}, \ref{t3.5} and \cite[Corollary 1.22]{Kirsch08}) we immediately obtain the following important result.

\begin{theorem}\label{thm3.8}
The ranges of $G_D$ and $(\Lambda_D\!-\!\Lambda_0)^{1/2}$ coincide, i.e., ${\rm Ran}\,G_D\!=\!{\rm Ran}\,(\Lambda_D\!-\!\Lambda_0)^{1/2}$.
\end{theorem}

For a numerical implementation of factorization method, we need the following theorem.

\begin{theorem}\label{FMDtNtest}
Let $z\in B$. Then $\pa_\nu K(\cdot,z)|_{\pa B}\in{\rm Ran}\,G_D$ if and only if $z\in D$.
\end{theorem}

\begin{proof}
{\color{xxx}If $z\in D$, then $K(\cdot,z)$ is harmonic in $B\ba\{z\}$ and $K(\cdot,z)|_{\pa D}\in H^{1/2}(\pa D)$.
Therefore, it follows from $K(\cdot,z)=0$ on $\pa B$ that $\pa_\nu K(\cdot,z)|_{\pa B}=G_D(K(\cdot,z)|_{\pa D})\in{\rm Ran}\,G_D$.}

If $\pa_\nu K(\cdot,z)\in{\rm Ran}\,G_D$, then there exists $g\in H^{1/2}(\pa D)$ such that {\color{xxx}$\pa_\nu K(\cdot,z)|_{\pa B}=G_Dg=\pa_\nu w$ on $\pa B$, where $w\in H^1(B\ba\ov{D})$ solves (\ref{7})--(\ref{9}).
Noting that $K(\cdot,z)=0=w$ on $\pa B$, we conclude from Holmgren's theorem (see \cite[Theorem 6.7]{Kress14}) that $K(\cdot,z)=w$ in $B\ba\ov{D}$.}
Due to the singularity of $K(\cdot,z)$ at $z\in B$, we know $K(\cdot,z)=w\in H^1(B\ba\ov{D})$ if and only if $z\in D$.
\end{proof}

Combining Remark \ref{0915} (ii), Theorems \ref{thm3.8} and \ref{FMDtNtest}, and Picard's theorem (\cite[Theorem 4.8]{CK19}), we immediately obtain the following result.

\begin{theorem}
{\color{xxx}Denote by $(\lambda_n,\varphi_n)$ an eigensystem of $\Lambda_D-\Lambda_0$.
Define
\be\label{0409-1}
I(z):=\left[\sum_n\frac{|(\pa_\nu K(\cdot,z),\varphi_n)|^2}{|\lambda_n|}\right]^{-1},\quad z\in B,
\en
where $(\cdot,\cdot)$ denotes the inner product in $L^2(\pa B)$.
Then $z\in D$ if and only if $I(z)>0$.}
\end{theorem}

We remark that, the knowledge of the operators $\Lambda_D$ and $\Lambda_0$ is equivalent with the Cauchy data $(f,\pa_\nu u|_{\pa B})$ and $(f,\pa_\nu u_0|_{\pa B})$ {\color{hgh} for sufficiently many} input data $f$, where $u$ solves (\ref{1'})--(\ref{3'}) and $u_0$ solves (\ref{5})--(\ref{6}). {\color{hgh}The measurement} data set can be used to reconstruct the shape of a general obstacle $D$ (e.g., \cite[Chapter 6]{Kirsch08} and Example \ref{FMDtN} below).
In the next section, we will show that, when $D$ is a convex polygon, both $\partial D$ and $\gamma$ can be reconstructed from a single pair of Cauchy data.

\section{Uniqueness with a single Cauchy data}\label{s3}
\setcounter{equation}{0}

This section is devoted to the uniqueness results and numerical methods of the inverse problem of determining the constant conductivity coefficient $\gamma$ in (\ref{1}) and the boundary $\partial D$ from a single pair of Cauchy data to (\ref{1})--(\ref{3}).

\subsection{Reconstruct the constant conductivity coefficient.}

Let $\Omega$ be a Lipschitz bounded domain that contained in the interior of $B$.
Denote by $G_{\Omega}$ the operator $G_D$ with $D$ replaced by $\Omega$ and by ${\rm Ran}\,G_{\Omega}$ the range of $G_{\Omega}$.

\begin{theorem}\label{thm4.1}
Let $u$ solve (\ref{1})--(\ref{3}) and $u_0$ solve (\ref{5})--(\ref{6}) with the same boundary value $f$ on $\pa B$.
Assume that $D\subset\Omega\subset\ov{\Om}\subset B$.
For $\tau>0$ define $g_\tau:=\left.(\gamma\pa_\nu u-\tau\pa_\nu u_0)\right|_{\pa B}$.
Then the following statements are true:

(i) If $\tau=\gamma$, then $g_\tau\in{\rm Ran}\,G_{\Omega}$;

(ii) If $\tau\neq\gamma$, then $g_\tau\in{\rm Ran}\,G_{\Omega}$ if and only if $\Lambda_0f\in{\rm Ran}\,G_\Om$.
\end{theorem}

\begin{proof}
(i). If $\tau=\gamma$, then $g_\tau=g_\gamma=\gamma\left.\pa_\nu(u-u_0)\right|_{\pa B}$.
Noting that $u-u_0$ is harmonic in $B\ba\ov{D}$ and $u-u_0=f-f=0$ on $\pa B$, we have $g_\tau=G_{\Omega}\left[\left.\gamma(u-u_0)\right|_{\pa \Omega}\right]$.

(ii). Since $g_\gamma\in{\rm Ran}\,G_{\Omega}$, the statement follows easily from
\ben
g_\tau-g_\ga=(\gamma\pa_\nu u-\tau\pa_\nu u_0)-(\gamma\pa_\nu u-\gamma\pa_\nu u_0)=(\gamma-\tau)\pa_\nu u_0=(\gamma-\tau)\Lambda_0f\quad\text{on }\pa B.
\enn
\end{proof}
According to Theorem \ref{thm4.1}, the constant conductivity coefficient $\gamma$ can be uniquely determined by a single pair of Cauchy data $(f,\pa_\nu u|_{\pa B})$ to \eqref{1}--\eqref{3} provided $\Lambda_0f\notin{\rm Ran}\,G_\Om$.
Moreover, based on Theorem \ref{thm4.1} we can propose a numerical approach for recovering $\gamma$ by taking $\tau>0$ as a testing parameter.
Our method for recovering $\gamma$ consists of the following three steps:

Step 1: Find a Lipschitz domain $\Om$ such that $D\subset\Omega\subset\ov{\Om}\subset B$ and calculate ${\rm Ran}\,G_\Om$;

Step 2: Find a boundary value $f$ on $\pa B$ such that $\Lambda_0f\notin{\rm Ran}\,G_\Om$;

Step 3: Test the values of $\tau>0$ till {\color{xxx}$g_\tau\in{\rm Ran}\,G_\Om$}.

\begin{remark}\label{rem1029-1}
Instead of by calculating $G_\Om$ directly, we obtain ${\rm Ran}\,G_{\Omega}$ indirectly by calculating $\Lambda_0$ and $\Lambda_\Omega$ {\color{xxx}(i.e., $\Lambda_D$ with $D$ replaced by $\Omega$)}.
{\color{xxx}More precisely, we have ${\rm Ran}\,G_\Om={\rm Ran}\,(\Lambda_\Om-\Lambda_0)^{1/2}$} (see Theorem \ref{thm3.8}).
Noting that $\Lambda_\Om$, $\Lambda_0$ and $g_\tau$ are equivalent to corresponding Cauchy data, we are able to recover the coefficient $\gamma$ in a data-to-data manner.
\end{remark}
{\color{hgh}For the purpose of recovering $\gamma$ in Step 3, we employ} the indicator function
\be\label{I1}
I_1(\tau):=\left[\sum_n\frac{|(g_\tau,f_n)|^2}{|\lambda_n|}\right]^{-1},
\en
where $(\cdot,\cdot)$ denotes {\color{xxx}the inner product in $L^2(\pa B)$} and $(\lambda_n,f_n)$ is an eigensystem of $(\Lambda_{\Omega}-\Lambda_0)$.
In view of Theorem \ref{thm4.1} and Remark \ref{rem1029-1}, we conclude from Picard's theorem (see \cite[Theorem 4.8]{CK19}) and Remark \ref{0915} (ii) that
\ben
I_1(\tau)\begin{cases}
>0 & \text{if }\tau=\gamma,\\
=0 & \text{if }\tau\neq\gamma.
\end{cases}
\enn

For convenience of numerical implementation, we provide several explicit examples of the boundary value $f$ on $\pa B$ such that $\Lambda_0f\notin{\rm Ran}\,G_\Om$, as shown {\color{hgh} in} the following theorem.
\begin{theorem}\label{thm240527}
Let $u$ solve (\ref{1})--(\ref{3}) and $u_0$ solve (\ref{5})--(\ref{6}) with the same boundary value $f$ on $\pa B$.
Assume that $D\subset\Omega\subset\ov{\Om}\subset B$.
For $\tau>0$, define $g_\tau:=\left.(\gamma\pa_\nu u-\tau\pa_\nu u_0)\right|_{\pa B}$.
Assume further that $\tau\neq\gamma$.
Then $\Lambda_0f\notin{\rm Ran}\,G_\Om$ provided one of the following conditions holds:

(i) $f$ is not identically zero but vanishes in an open subset $\Gamma$ of $\pa B$;

(ii) $f=\tilde f+c$ with $c$ being an arbitrary constant and $\tilde f$ satisfying condition (i);

(iii) $f\in H^s(\pa B)\ba H^{s+\epsilon}(\pa B)$ for any $s\geq-\frac12$ and $\epsilon>0$.
\end{theorem}
\begin{proof}
(i). {\color{xxx}Assume to the contrary that $\Lambda_0f\in{\rm Ran}\,G_\Om$.}
Then by Theorem \ref{thm4.1} (ii) it follows that $g_\tau\in{\rm Ran}\,G_{\Omega}$, i.e., there exists $g\in H^{1/2}(\pa \Omega)$ such that $g_\tau=G_{\Omega}g$.
By the definition of $G_\Omega$ the boundary value problem
\ben
\Delta w=0 & \text{in}\;B\ba\ov{\Omega},\\
w=0 & \text{on}\;\pa B,\\
w=g & \text{on}\;\pa \Omega,
\enn
admits a unique solution $w$ such that $\pa_\nu w=g_\tau$ on $\pa B$.
From (\ref{1})--(\ref{3}) and (\ref{5})--(\ref{6}) we obtain
\ben
\Delta(\gamma u-\tau u_0)=0 & \text{in}\;B\ba\ov{\Omega},\\
(\gamma u-\tau u_0)=(\gamma-\tau)f & \text{on}\;\pa B.
\enn
Since {\color{xxx}$\pa_\nu(\gamma u-\tau u_0)=g_\tau$ on $\pa B$ and} $f$ vanishes on $\Gamma$, the Cauchy data of $(\gamma u-\tau u_0)$ and $w$ coincide on $\G$.
Now it follows from Holmgren's theorem (see \cite[Theorem 6.7]{Kress14}) that $w=\gamma u-\tau u_0$ in $B\ba\ov{\Omega}$. Consequently,
$w$ and $\gamma u-\tau u_0$ must coincide on $\partial B$ by trace theorem.
However, this is impossible, because $w=0$ on $\pa B$ and $(\gamma u-\tau u_0)=(\gamma-\tau)f$ is not identically zero on $\pa B$.

(ii). It suffices to show $\Lambda_0(f-c)\notin{\rm Ran}\,G_\Om$ since constant functions belong to the nullspace of $\Lambda_0$.
Now the proof is completed by statement (i).

(iii). Without loss of generality, we assume that $B$ is a disk with radius $R$ and center at the origin.
From the following Fourier series expansion
\ben
f(x)=f(R(\cos\theta,\sin\theta))=\sum_{n\in\Z}c_ne^{in\theta},\quad x\in\pa B,
\enn
we conclude that
\be\label{0419-1}
\|f\|_{H^s(\pa B)}^2:=\sum_{n\in\Z}(1+n^2)^s|c_n|^2<\infty,
\en
but the series
\be\label{0419-2}
\sum_{n\in\Z}(1+n^2)^{s+\epsilon}|c_n|^2
\en
diverges.
Analogously to \cite[Theorem 2.22]{KH}, {\color{xxx} it can be shown that the solution to} \eqref{5}--\eqref{6} is given by the series
\ben
u_0(x)=\sum_{n\in\Z}c_n\frac{|x|^{|n|}}{R^{|n|}}e^{in\theta},\quad x\in B.
\enn
Therefore,
\ben
(\Lambda_0f)(x)=\pa_\nu u_0(x)=\sum_{n\in\Z\ba\{0\}}|n|c_n\frac{|x|^{|n|-1}}{R^{|n|}}e^{in\theta},\quad x\in\pa B.
\enn

{\color{xxx}If $\sigma\leq s$, then $\Lambda_0f\in H^{\sigma-1}(\pa B)$ since from} \eqref{0419-1} we deduce
\ben
\|\Lambda_0f\|_{H^{\sigma-1}(\pa B)}^2=\sum_{n\in\Z\ba\{0\}}(1+n^2)^{\sigma-1}\frac{n^2|c_n|^2}{R^2}\leq\frac1{R^2}\sum_{n\in\Z\ba\{0\}}(1+n^2)^{\sigma}|c_n|^2\leq\frac1{R^2}\|f\|_{H^s(\pa B)}^2.
\enn

{\color{xxx}If $\sigma\geq s+\epsilon$, then $\Lambda_0f\notin H^{\sigma-1}(\pa B)$ since from} \eqref{0419-1} and \eqref{0419-2} we know the series
\ben
\sum_{n\in\Z\ba\{0\}}\!\!\!\!(1+n^2)^{\sigma-1}\frac{n^2|c_n|^2}{R^2}\!\!\geq\!\!\frac1{2R^2}\!\!\!\!\sum_{n\in\Z\ba\{0\}}\!\!\!\!(1+n^2)^\sigma|c_n|^2\geq\frac1{2R^2}\sum_{n\in\Z}(1+n^2)^{s+\epsilon}|c_n|^2-\frac{\|f\|_{H^s(\pa B)}^2}{2R^2}
\enn
diverges.
{\color{xxx}However,} it follows from the proof of Theorem \ref{t3.3} that ${\rm Ran}\,G_\Om\subset H^{m-3/2}(\pa B)$ for any $m\in\Z_+$.
Consequently, $\Lambda_0f\notin{\rm Ran}\,G_\Om$.
\end{proof}
\begin{remark}
{\color{xxx}Using the notations introduced in the above proof, an explicit example of condition (iii) in Theorem \ref{thm240527} is as follows.}
For any $\theta_0\in(0,2\pi)$ {\color{xxx}set}
\ben
f(R(\cos\theta,\sin\theta))=\begin{cases}
1, & \text{if }\theta\in[0,\theta_0],\\
0, & \text{if }\theta\in(\theta_0,2\pi).
\end{cases}
\enn
Then
\ben
c_n=\left(f,\frac{e^{in\theta}}{{\color{xxx}2\pi R}}\right)_{L^2(\pa B)}=\frac1{{\color{xxx}2\pi R}}\int_0^{\theta_0}e^{-in\theta}d\theta=\begin{cases}
\frac{\theta_0}{{\color{xxx}2\pi R}}, & n=0,\\
\frac{1-e^{-in\theta_0}}{{\color{xxx}2\pi R} in}, & n\neq0.
\end{cases}
\enn
It can be easily seen that $f\in H^s(\pa B)\ba H^{1/2}(\pa B)$ provided $s<\frac12$.
We also have $f+f_0\in H^s(\pa B)\ba H^{1/2}(\pa B)$ provided $s<\frac12$ and $f_0\in H^{1/2}(\pa B)$.
\end{remark}

\subsection{Reconstruct the boundary of the obstacle.}
Having determined $\gamma$ in the previous subsection, we shall proceed with the inverse problem of finding $\partial D$.

\begin{theorem}
Assume that the constant coefficient $\gamma>0$ is known and $D$ is a Dirichlet domain contained in $B$.
Then $D$ can be uniquely determined by the single pair of Cauchy data $(f,\gamma\pa_\nu u|_{\pa B})$ to (\ref{1})--(\ref{3}) provided $f$ is not identically zero.
\end{theorem}

\begin{proof}
Following the proof of \cite[Theorem 5.1]{CK19}, we {\color{xxx}assume that $D_1\neq D_2$ and} the Cauchy data corresponding to \eqref{1}--\eqref{3} with $D=D_1$ and $D=D_2$ coincide.
Let $G$ be the component of $B\ba\{\ov{D_1}\cup\ov{D_2}\}$ whose boundary contains $\pa B$.
Without loss of generality, we assume $D^*:=(B\ba G)\ba\ov{D_2}$ is nonempty.
For $j=1,2$, let $u_j$ solve
\ben
{\rm div}(\gamma\nabla u_j)=0 & \text{in}\;B\ba\ov{D_j},\\
u_j=f & \text{on}\;\pa B,\\
u_j=0 & \text{on}\;\pa D_j.
\enn
Suppose $u_1$ and $u_2$ have the same Cauchy data on $\pa B$, i.e., $(f,\gamma\pa_\nu u_1|_{\pa B})=(f,\gamma\pa_\nu u_2|_{\pa B})$.
It follows from Holmgren's theorem (see \cite[Theorem 6.7]{Kress14}) that {\color{xxx}$u_1=u_2$ in $G$}.
Since $u_j=0$ on $\pa D_j$, $j=1,2$, we know that {\color{xxx}$u_2=0$} on $\pa D^*$.
By the Maximum-Minimum Principle of harmonic functions (see \cite[Corollary 6.10]{Kress14}), we have {\color{xxx}$u_2=0$} in $D^*$.
Hence, {\color{xxx}$u_2=0$} in $G$ by analyticity.
This leads to a contradiction since {\color{xxx}$f=u_2|_{\pa B}$ is not identically zero.}
\end{proof}
\begin{remark}\label{rem3.6}
If $f$ is not identically zero, the solution $u$ to \eqref{1'}--\eqref{3'} cannot be analytically extended from $B\ba\ov{D}$ into $B$.
Actually, if $u$ can be analytically extended into $B$, then $u$ is harmonic in $B$.
Proceeding as above, we can conclude from $u=0$ on $\pa D$ that $u$ vanishes identically in $B$.
This is a contradiction to the assumption on $f$.
\end{remark}
In what follows we shall design a domain-defined sampling method for imaging a convex polygonal obstacle $D$.
Below we show that the analytical extension across a corner of $\partial D$ is impossible.

\begin{theorem}\label{thm2.1}
Assume that the constant coefficient $\gamma>0$ is known, $f$ is not identically zero, and $D$ is a Dirichlet convex polygon.
Assume further that one of the following statement holds:

(i) All the corners of $D$ are irrational angles;

(ii) ${\rm dist}(\pa B,D)>{\rm diam}(D)$;

(iii) $f\in C(\pa B)$ and $f(x)\neq0$ for all $x\in\pa B$.

Then $u$ cannot be analytically extended across the corners of $D$.
\end{theorem}
\begin{proof}
For the proof of statements (i) and (ii), we refer the reader to \cite[Lemma 2.4]{WHB23}.

It remains to show statement (iii).
Since $f\in C(\pa B)$, one can show $u\in C^2(B\ba\ov{B_0})\cap C(\ov{B}\ba B_0)$, where $B_0$ is an open subset of $B$ such that $\ov{D}\subset B_0\subset\ov{B_0}\subset B$.
Assume to the contrary that $u$ can be analytically extended across a corner $z$ of $D$.
Consider an edge $\Gamma$ of $D$ that has $z$ as one of its endpoints.
Since $u$ vanishes on $\Gamma$ and analytic in a neighborhood of $z$ and $B\ba\ov{D}$, we conclude that $u$ vanishes on the extended line of $\Gamma$ that containing $z$ as an interior point.
Because of convexity, the extended line must intersect with $\pa B$ at a point $x_0$.
Therefore, $f(x_0)\!=\!u(x_0)\!=\!0$.
This is a contradiction to the assumption on $f$.
\end{proof}
By Theorem \ref{thm2.1} we can characterize $D$ by sampling domains $\Omega$ and the testing function $g_\gamma$.
\begin{theorem}\label{thm211027}
Let $u$ solve (\ref{1})--(\ref{3}) and $u_0$ solve (\ref{5})--(\ref{6}) with the same boundary value $f$ on $\pa B$.
Set $g_\gamma:=\left.(\gamma\pa_\nu u-\gamma\pa_\nu u_0)\right|_{\pa B}$.
Let the assumptions of Theorem \ref{thm2.1} be fulfilled.
Assume further that $\Omega$ is a convex domain.
Then $D\subset\Omega$ if and only if $g_\gamma\in{\rm Ran}\,G_{\Omega}$.
\end{theorem}

\begin{proof}
If $D\subset\Omega$, then $(B\ba\ov{\Omega})\subset(B\ba\ov{D})$ and thus $g_\gamma=G_{\Omega}\left[\gamma(u-u_0)|_{\pa \Omega}\right]$ on $\pa B$.

Now, we consider the case when $D\not\subset\Omega$.
Assume to the contrary that {\color{xxx}$\left.\pa_\nu(u-u_0)\right|_{\pa B}\in{\rm Ran}\,G_{\Omega}$}, i.e., there exists $g\in H^{1/2}(\pa \Omega)$ such that $g_\gamma=G_{\Omega}g$ on $\pa B$.
Therefore, the unique solution $w$ of the boundary value problem
\ben
\Delta w=0 & \text{in}\;B\ba\ov{\Omega},\\
w=0 & \text{on}\;\pa B,\\
w=g & \text{on}\;\pa \Omega,
\enn
must satisfy $\pa_\nu w=g_\gamma$ on $\pa B$.
From (\ref{1})--(\ref{3}) and (\ref{5})--(\ref{6}) we obtain
\ben
\Delta(u-u_0)=0 & \text{in}\;B\ba\ov{\Omega},\\
(u-u_0)=0 & \text{on}\;\pa B.
\enn
{\color{xxx}Since $\gamma\pa_\nu(u-u_0)=g_\gamma$ on $\pa B$,} the Cauchy data of $\gamma(u-u_0)$ and $w$ coincide on $\pa B$.
It follows from Holmgren's theorem (see \cite[Theorem 6.7]{Kress14}) that
\be\label{211027-1}
w=\gamma(u-u_0)\quad\text{in }B\ba(\ov{D}\cup\ov{\Omega}).
\en
Noting that $D\not\subset\Omega$ and $\Omega$ is convex, we can find a corner $z$ of $D$ and an open neighbourhood $V$ of $z$ such that $V$ dose not intersect with the closure of $\Omega$.
We deduce from (\ref{211027-1}) that $u$ can be extended as a harmonic function in $V$, which contradicts the result of Theorem \ref{thm2.1}.
\end{proof}

\begin{remark}\label{rem1029-2}
Denote by $\mathcal O$ the set of all convex Lipschitz domains that contained {\color{xxx}in} $B$.
Let the assumptions of Theorem \ref{thm2.1} be fulfilled.
By Theorems \ref{thm3.8} and \ref{thm211027} we have
\ben
D=\bigcap_{\Om\in\{\widetilde\Omega\in\mathcal O:g_{\gamma}\in{\rm Ran}\,G_{\widetilde\Omega}\}}\Omega=\bigcap_{\Om\in\{\widetilde\Omega\in\mathcal O:g_{\gamma}\in{\rm Ran}\,(\Lambda_{\widetilde\Omega}\!-\!\Lambda_0)^{1/2}\}}{\color{xxx}\Omega.}
\enn
\end{remark}

{\color{xxx}According to} the above analysis (see also \cite[Theorems 6.2 and 6.3]{WHB21}), we defined the following indicator function to recover the location and rough profile of $D$:
\be\label{I2}
I_2(\Omega):=\left[\sum_n\frac{|(g _\gamma,f_n)|^2}{|\lambda_n|}\right]^{-1},
\en
where $\Omega\in\mathcal O$ with $\mathcal O$ introduced in Remark \ref{rem1029-2}, $(\cdot,\cdot)$ denotes the inner product in $L^2(\pa B)$ and $(\lambda_n,f_n)$ is an eigensystem of $(\Lambda_{\Omega}-\Lambda_0)$.
Under the assumptions of Theorems \ref{thm2.1} and \ref{thm211027}, we conclude from Remark \ref{rem1029-1}, Picard's theorem (see \cite[Theorem 4.8]{CK19}) and Remark \ref{0915} (ii) that
\ben
I_2(\Omega)\begin{cases}
>0 & \text{if }D\!\subset\!\Om,\\
=0 & {\color{xxx}\text{if }D\!\not\subset\!\Om}.
\end{cases}
\enn
Moreover, in view of Remark \ref{rem3.6}, a rough location of $D$ can be recovered by \eqref{I2} even if the assumptions of Theorem \ref{thm2.1} cannot be fulfilled.

\section{Numerical simulation of forward problems}\label{s3.5}
\setcounter{equation}{0}

We note that the boundary value problems (\ref{1})--(\ref{3}) (including (\ref{1'})--(\ref{3'})) and (\ref{5})--(\ref{6}) can be reduced to boundary integral equations (see \cite[Chapter 6]{Kress14}), which can be numerically calculated by Nystr\"om's method (see \cite[Chapter 12]{Kress14} and \cite[Section 3.6]{CK19}).
To calculate $\Lambda_0$, $\Lambda_\Om$ and $\Lambda_D$, it suffices to represent these Dirichlet-to-Neumann operators in terms of boundary integral boundary operators introduced at the beginning of Section \ref{s2}.

\begin{theorem}\label{0913}
Let $f\in H^s(\pa B)$ with $s\in[-\frac12,\frac12]$.
The double-layer potential $u_0:=\mathcal D_{0,\pa B}\varphi_{\pa B}$ with density $\varphi_{\pa B}\in H^s(\pa B)$ is a solution of the Dirichlet boundary value problem (\ref{5})--(\ref{6}), provided that $\varphi_{\pa B}$ is a solution of the boundary integral equation
\be\label{0913-5}
(K_{0,\pa B}-I)\varphi_{\pa B}=2f,
\en
where $I$ denotes the identity operator.
Furthermore, (\ref{0913-5}) is uniquely solvable for all $f\in H^s(\pa B)$ and the Dirichlet-to-Neumann operator $\Lambda_0$ can be represented by
\ben
\Lambda_0=T_{0,\pa B}(K_{0,\pa B}-I)^{-1}.
\enn
Moreover, $\Lambda_0:H^s(\pa B)\ra H^{s-1}(\pa B)$ is bounded.
\end{theorem}

\begin{proof}
By denseness arguments as in the proof of Theorem \ref{0404-3}, the jump relations and regularity properties for the boundary trace and the normal derivative trace of the single- and double-layer potential remain valid in the Sobolev space setting (see \cite{Costabel,WM}).
Therefore, $K_{0,\pa B}:H^s(\pa B)\ra H^{s+1}(\pa B)$ is bounded, and thus $K_{0,\pa B}:H^s(\pa B)\ra H^s(\pa B)$ is compact.
Proceeding as in the proofs of \cite[Theorems 6.22 and 6.23]{Kress14}, {\color{xxx}we can prove this theorem}.

Moreover, it follows from \cite[Theorem 3.6 and Corollary 10.13]{CK19} and the interpolation property of the Sobolev spaces (see \cite[Theorem 8.13]{Kress14}) that $T_{0,\pa B}:H^s(\pa B)\ra H^{s-1}(\pa B)$ is bounded.
In view of Theorem \ref{0404-4}, Remark \ref{0413-1} (iv), by the denseness argument {\color{xxx}as} in the proof of Theorem \ref{0404-3} and jump relations, we conclude from \cite[Theorem 10.12]{CK19} that $\Lambda_0=T_{0,\pa B}(K_{0,\pa B}-I)^{-1}$ and $\Lambda_0:H^s(\pa B)\ra H^{s-1}(\pa B)$ is bounded.
\end{proof}
\begin{theorem}\label{0914}
Let $f\in H^s(\pa B)$ with $s\in[-\frac12,\frac12]$.
The combined layer potential
\be\label{0914-6}
u=\mathcal D_{0,\pa B}\varphi_{\pa B}+\mathcal S_{0,\pa D}\varphi_{\pa D}
\en
with densities $\varphi_{\pa B}\in H^s(\pa B)$ and $\varphi_{\pa D}\in H^{-1/2}(\pa D)$ is a solution of the Dirichlet problem (\ref{1'})--(\ref{3'}) provided that $\varphi_{\pa B}$ and $\varphi_{\pa D}$ satisfy the boundary integral equations
\be\label{0913-4}
\left(\begin{array}{cc}
K_{0,\pa B}-I & S_{0,\pa D\ra\pa B}\\
K_{0,\pa B\ra\pa D} & S_{0,\pa D}
\end{array}\right)\left(\begin{array}{c}
\varphi_{\pa B}\\
\varphi_{\pa D}
\end{array}\right)=\left(\begin{array}{c}
2f\\
0
\end{array}\right).
\en
Furthermore, (\ref{0913-4}) is uniquely solvable for all $f\in H^s(\pa B)$ and {\color{xxx}the Dirichlet-to-Neumann operator $\Lambda_D$} can be represented by
\be\label{0507}
\Lambda_Df=\left(\begin{array}{cc}
T_{0,\pa B} & K'_{0,\pa D\ra\pa B}
\end{array}\right)\left(\begin{array}{cc}
K_{0,\pa B}-I & S_{0,\pa D\ra\pa B}\\
K_{0,\pa B\ra\pa D} & S_{0,\pa D}
\end{array}\right)^{-1}\left(\begin{array}{c}
f\\
0
\end{array}\right).
\en
Moreover, $\Lambda_D:H^s(\pa B)\ra H^{s-1}(\pa B)$ is bounded.
\end{theorem}

\begin{proof}
By the jump relations and regularity properties in the Sobolev space setting (see \cite{Costabel,WM}), it is easy to see that the combined layer potential (\ref{0914-6}) with densities $\varphi_{\pa B}$ and $\varphi_{\pa D}$ satisfying (\ref{0913-4}) is a solution to the Dirichlet problem (\ref{1'})--(\ref{3'}).

From the proof of Theorem \ref{0913} we know $K_{0,\pa B}-I:H^s(\pa B)\ra H^s(\pa B)$ has a bounded inverse. {\color{hgh} The operator}
$S_{i,\pa D}:H^{-1/2}(\pa D)\ra H^{1/2}(\pa D)$ also has a bounded inverse (see \cite[Theorem 5.44]{KH}).
Moreover, $S_{0,\pa D\ra\pa B}:H^{-1/2}(\pa D)\ra H^s(\pa B)$ and $K_{0,\pa B\ra\pa D}:H^s(\pa B)\ra H^{1/2}(\pa D)$ are compact due to the regularity properties of surface potentials, and $S_{0,\pa D}-S_{i,\pa D}:H^{-1/2}(\pa D)\ra H^{1/2}(\pa D)$ is also compact {\color{xxx}due to the increased smoothness of the integral kernel as compared with that of $S_{i,\pa D}$}.
By the Riesz-Fredholm theory we conclude from
\be\label{0414-2}
\left(\begin{array}{cc}
K_{0,\pa B}-I & S_{0,\pa D\ra\pa B}\\
K_{0,\pa B\ra\pa D} & S_{0,\pa D}
\end{array}\right)=\left(\begin{array}{cc}
K_{0,\pa B}-I & 0\\
0 & S_{i,\pa D}
\end{array}\right)+\left(\begin{array}{cc}
0 & S_{0,\pa D\ra\pa B}\\
K_{0,\pa B\ra\pa D} & S_{0,\pa D}-S_{i,\pa D}
\end{array}\right)
\en
that (\ref{0913-4}) is uniquely solvable if and only if (\ref{0414-2}) is injective.

{\color{xxx}Let $u$ be given by \eqref{0914-6} with $\varphi_{\pa B}$ and $\varphi_{\pa D}$ satisfy \eqref{0913-4} wtih $f=0$ on $\pa B$.}
By the uniqueness for the interior Dirichlet problem in $B\ba\ov{D}$ and in $D$, respectively, we have $u=0$ in $B\ba\ov{D}$ and in $D$, respectively.
By the jump relation we have
\be\label{0914-8}
\varphi_{\pa D}=\pa_\nu u_--\pa_\nu u_+=0\quad\text{on}\;\pa D.
\en
Therefore, {\color{xxx}it follows from} (\ref{0914-6}) that
\be\label{0915-1}
u=\mathcal D_{0,\pa B}\varphi_{\pa B}.
\en
The jump relation now yields $\pa_\nu u_\pm=0$ on $\pa B$.
Note that (\ref{0915-1}) implies $u(x)=o(1)$ as $|x|\ra\infty$.
It follows from the uniqueness of the exterior Neumann problem (see \cite[Theorem 6.13]{Kress14}) that $u=0$ in $\R^2\ba\ov{B}$.
By the jump relation we have
\be\label{0914-7}
{\color{xxx}\varphi_{\pa B}=u_+-u_-}=0\quad\text{on}\;\pa B.
\en
Combining (\ref{0914-8}) and (\ref{0914-7}) we thus obtain the injectivity of (\ref{0414-2}).

Finally, by noting that
\ben
\pa_\nu u_0=\frac12\left(\begin{array}{cc}
T_0\varphi_{\pa B}& K'_{0,\pa D\ra\pa B}
\end{array}\right)\left(\begin{array}{c}
\varphi_{\pa B}\\
\varphi_{\pa D}
\end{array}\right)\quad\text{on}\;\pa B,
\enn
we arrive at \eqref{0507}.
Since {\color{xxx}$K'_{0,\pa D\ra\pa B}:H^{-1/2}(\pa D)\ra H^{s-1}(\pa B)$} and $T_0:H^{s}(\pa B)\ra H^{s-1}(\pa B)$ are bonuded, we know $\Lambda_D:H^s(\pa B)\ra H^{s-1}(\pa B)$ is bounded.
\end{proof}

\begin{remark}
(i) In contrast to the invertible operator $S_{\pa D}:H^{-1/2}(\pa D)\ra H^{1/2}(\pa D)$ defined in Theorem \ref{t+}, the operator $S_{0,\pa D}:H^{-1/2}(\pa D)\ra H^{1/2}(\pa D)$ is in general not injective in two dimensions (see \cite[Theorem 7.38]{Kress14}).
A modified operator of $S_{0,\pa D}$ has been proposed to overcome this difficulty (see \cite[Theorem 7.41]{Kress14}).
Here, we avoid this difficulty by using the combined layer potential (\ref{0914-6}).

(ii) {\color{xxx}Let $s\in[-\frac12,\frac12]$.}
By Theorems \ref{0913} and \ref{0914}, we immediately obtain $(\Lambda_D-\Lambda_0)f\in H^{s-1}(\pa B)$ for all $f\in H^s(\pa B)$.
Furthermore, Remark \ref{0915} (i) implies {\color{xxx}$(\Lambda_D-\Lambda_0)f\in H^{m-3/2}(\pa B)$ for all $f\in H^s(\pa B)$ and $m\in\Z_+$.}
\end{remark}

{\color{hgh}Below we will consider numerical approaches for computing $\Lambda_0$ and $\Lambda_D$.
The same scheme can be applied to {\color{xxx}handle} $\Lambda_{\Omega}$ for any Lipschitz domain $\Omega$ satisfying {\color{xxx}$\ov{\Om}\subset B$}.}
We begin with the necessary parametrization of \eqref{0913-5} and \eqref{0913-4}.
{\color{xxx}Note that a uniform mesh for $\pa D$ yields only poor convergence due to the corners of polygon $D$.}
To take proper care of the corner singularities, we may use a graded mesh as shown in (\ref{1109-1}) below (for details see \cite[Section 3.6]{CK19} and \cite{QZZ}).
Assume that $\pa B$ possesses a regular analytic $2\pi$-periodic representation of the form
\be\label{1110-5}
\tilde x(\tilde t)=(\tilde x_1(\tilde t),\tilde x_2(\tilde t)),\quad0\leq\tilde t\leq2\pi.
\en
Moreover, the boundary $\pa D$ of polygon $D$ with corners given in order by $\{P_\ell:=(P_{\ell,1},P_{\ell,2})\}_{\ell=1}^N$ possesses a parametric representation of the form $x(t):=(x_1(t),x_2(t))$ for $0\leq t\leq2\pi$, where
\be\label{1113}
{\color{xxx}x_j(t)\!=\!\left(\!\frac{2\ell\pi}N\!-\!t\!\right)\!\!P_{\ell,j}\!+\!\left(\!t\!-\!\frac{2(\ell\!-\!1)\pi}N\!\right)\!\!P_{(\ell\,{\rm mod}\,N)+1,j},\;t\!\in\!\left[\!\frac{2(\ell\!-\!1)\pi}N,\frac{2\ell\pi}N\!\right),\ell\!=\!1,\!\cdots\!,N}
\en
for $j=1,2$.
To introduce the uniform mesh on $\pa B$ and the graded mesh on $\pa D$, we choose $\tilde n,n\in\Z_+$ such that $n/N\in\Z_+$.
For simplicity let there be $n/N$ knots on each smooth segment.
The knots on $\pa B$ and $\pa D$ are $\{\tilde x(\tilde t_j)\}$ and $\{x(t_j)\}$, respectively, where
\be\no
&&\tilde t_j=\frac{j\pi}{\tilde n},\;j=0,1,\cdots,2\tilde n-1,\\ \label{1109-1}
&&t_j=w(s_j)\text{ with }s_j=\frac{\pi}{2n}\!+\!\frac{j\pi}{n},\;j=0,1,\cdots,2n-1,\\ \no
&&w(s)=\tilde w(Ns-2\ell\pi+2\pi),\;
{\color{xxx}s\in\left[\frac{2(\ell-1)\pi}{N},\frac{2\ell\pi}{N}\right)},\ell=1,\cdots,N,\\ \no
&&\tilde w(s)=2\pi\frac{[v(s)]^p}{[v(s)]^p+[v(2\pi-s)]^p},\quad 0\leq s\leq2\pi,\\ \no
&&v(s)=\left(\frac1p-\frac12\right)\left(\frac{\pi-s}\pi\right)^3+\frac1p\frac{s-\pi}\pi+\frac12,\quad p\geq2.
\en
Following the Nystr\"om method (see \cite[Chapter 12]{Kress14}) and the idea of graded mesh (see \cite[Section 3.6]{CK19}), the boundary integral equations (\ref{0913-5}) and (\ref{0913-4}) can be approximated by
\ben
(L_{\pa B}-I)\Psi_{\pa B}=F
\enn
and
\be\label{1110-1}
\left(\begin{array}{cc}
L_{\pa B}-I & M_{\pa D\ra\pa B}\\
L_{\pa B\ra\pa D} & M_{\pa D}
\end{array}\right)\left(\begin{array}{c}
\Psi_{\pa B}\\
W\Psi_{\pa D}
\end{array}\right)=\left(\begin{array}{c}
F\\
0
\end{array}\right),
\en
respectively, where $I$ denotes the identity matrix and\
\ben
L_{\pa B}&:=&\left(\frac{\pi}{\tilde n}\widetilde L_{\pa B}(\tilde t_i,\tilde t_j)\right)_{i,j=0,1,\cdots,2\tilde n-1},\\
\widetilde L_{\pa B}(\tilde t,\tilde\tau)&:=&\begin{cases}
\frac1\pi\frac{[\tilde x(\tilde t)-\tilde x(\tilde\tau)]\cdot(\tilde x'_2(\tilde\tau),-\tilde x'_1(\tilde\tau))}{|\tilde x(\tilde t)-\tilde x(\tilde\tau)|^2}, & \tilde t\neq\tilde\tau,\\
\frac{\tilde x''(\tilde t)\cdot(\tilde x'_2(\tilde t),-\tilde x'_1(\tilde t))}{2\pi|\tilde x'(\tilde t)|^2}, & \tilde t=\tilde\tau,
\end{cases}\\
\Psi_{\pa B}&:=&\left(\varphi_{\pa B}(\tilde x(\tilde t_0)),\varphi_{\pa B}(\tilde x(\tilde t_1)),\cdots,\varphi_{\pa B}(\tilde x(\tilde t_{2\tilde n-1}))\right)^\top,\\
F&:=&2\left(f(\tilde x(\tilde t_0)),f(\tilde x(\tilde t_1)),\cdots,f(\tilde x(\tilde t_{2\tilde n-1}))\right)^\top,
\enn
and
\ben
M_{\pa D\ra\pa B}&:=&\left(\frac{\pi}n\widetilde M_{\pa D\ra\pa B}(\tilde t_i,\tau_j)\right)_{i=0,1,\cdots,2\tilde n-1,j=0,1,\cdots,2n-1},\\
\widetilde M_{\pa D\ra\pa B}(\tilde t,\tau)&:=&\frac{|x'(\tau)|}\pi\ln\frac1{|\tilde x(\tilde t)-x(\tau)|},\\
L_{\pa B\ra\pa D}&:=&\left(\frac{\pi}{\tilde n}\widetilde L_{\pa B\ra\pa D}(t_i,\tilde t_j)\right)_{i=0,1,\cdots,2n-1,j=0,1,\cdots,2\tilde n-1},\\
\widetilde L_{\pa B\ra\pa D}(t,\tilde\tau)&:=&\frac1\pi\frac{[x(t)-\tilde x(\tilde\tau)]\cdot(\tilde x'_2(\tilde\tau),-\tilde x'_1(\tilde\tau))}{|x(t)-\tilde x(\tilde\tau)|^2},\\
M_{\pa D}&:=&\left(R_j^{(n)}(s_i)\widetilde M_{\pa D,1}(s_i,\sigma_j)+\frac{\pi}n\widetilde M_{\pa D,2}(s_i,\sigma_j)\right)_{i=0,1,\cdots,2n-1},\\
R_j^{(n)}(s)&:=&-\frac{2\pi}n\sum_{m=1}^{n-1}\frac 1m\cos m(s-s_j)-\frac\pi{n^2}\cos n(s-s_j),\quad j=0,1,\cdots,2n-1,\\
\widetilde M_{\pa D,1}(s,\sigma)&:=&M_{\pa D,1}(w(s),w(\sigma)),\\
\widetilde M_{\pa D,2}(s,\sigma)&:=&\begin{cases}
M_{\pa D,2}(w(s),w(\sigma)), & s\neq\sigma,\\
2M_{\pa D,1}(w(s),w(s))\ln w'(s)+M_{\pa D,2}(w(s),w(s)), & s=\sigma,
\end{cases}\\
M_{\pa D,1}(t,\tau))&:=&-\frac1{2\pi}|x'(\tau)|,\\
M_{\pa D,2}(t,\tau))&:=&\begin{cases}
\frac{|x'(\tau)|}\pi\ln\frac1{|x(t)-x(\tau)|}-M_{\pa D,1}(t,\tau)\ln\left(4\sin^2\frac{t-\tau}2\right), & t\neq\tau,\\
\frac{|x'(t)|}\pi\ln\frac1{|x'(t)|}, & t=\tau,
\end{cases}\\
W&:=&{\rm diag}(w'(s_0),w'(s_1),\cdots,w'(s_{2n-1})),\\
\Psi_{\pa D}&:=&\left(\varphi_{\pa D}(x(t_0)),\varphi_{\pa D}(x(t_1)),\cdots,\varphi_{\pa D}(x(t_{2n-1}))\right)^\top.
\enn
With the notations in discrete form, the solution to \eqref{5}--\eqref{6} can be approximated by
\ben
u_0(x)\approx\frac\pi{\tilde n}\left(\begin{array}{cccc}
\frac{\pa\Phi_0(x,\tilde x(\tilde t_0))}{\pa\nu(\tilde x(\tilde t_0))} & \frac{\pa\Phi_0(x,\tilde x(\tilde t_1))}{\pa\nu(\tilde x(\tilde t_1))} & \cdots & \frac{\pa\Phi_0(x,\tilde x(\tilde t_{2\tilde n-1}))}{\pa\nu(\tilde x(\tilde t_{2\tilde n-1}))}
\end{array}\right)Y_{\pa B}\Psi_{\pa B},\quad x\in B,
\enn
and the solution to \eqref{1'}--\eqref{3'} can be approximated by
\ben
u(x)&\approx&\frac\pi{\tilde n}\left(\begin{array}{cccc}
\frac{\pa\Phi_0(x,\tilde x(\tilde t_0))}{\pa\nu(\tilde x(\tilde t_0))} & \frac{\pa\Phi_0(x,\tilde x(\tilde t_1))}{\pa\nu(\tilde x(\tilde t_1))} & \cdots & \frac{\pa\Phi_0(x,\tilde x(\tilde t_{2\tilde n-1}))}{\pa\nu(\tilde x(\tilde t_{2\tilde n-1}))}
\end{array}\right)Y_{\pa B}\Psi_{\pa B}\\
&&+\frac\pi n\left(\begin{array}{cccc}
\Phi_0(x,x(t_0)) & \Phi_0(x,x(t_1)) & \cdots & \Phi_0(x,x(t_{2n-1}))
\end{array}\right)Y_{\pa D}W\Psi_{\pa D},\quad x\in B\ba\ov{D},
\enn
where
\be\no
Y_{\pa B}&=&{\rm diag}(|\tilde x'(\tilde t_0)|,|\tilde x'(\tilde t_1)|,\cdots,|\tilde x'(\tilde t_{2\tilde n-1})|),\\ \label{0507-1}
Y_{\pa D}&=&{\rm diag}(|x'(t_0)|,|x'(t_1)|,\cdots,|x'(t_{2n-1})|).
\en
\begin{remark}\label{rem0507}
Since $w'(s_j)$ takes a very small value if the knot $x(w(s_j))$ is close to the corners of $\pa D$, it is not stable to calculate $\Psi_{\pa D}$ from \eqref{1110-1}.
From the above approximation for the solution $u$ to \eqref{1'}--\eqref{3'}, we see $W\Psi_{\pa D}$ can be viewed as an unknown vector and it is sufficient to calculate $W\Psi_{\pa D}$ from \eqref{1110-1}.
\end{remark}
In view of Remark \ref{rem0507}, the Dirichlet-to-Neumann operators $\Lambda_0$ and $\Lambda_D$ are approximated by
\ben
\Lambda_0\approx Y_{\pa B}^{-1}T_{\pa B}\Psi_{\pa B}=Y_{\pa B}^{-1}T_{\pa B}(L_{\pa B}-I)^{-1}\\
\enn
and
\be\no
\Lambda_Df&\approx&\frac12Y_{\pa B}^{-1}\left(\begin{array}{cc}
T_{\pa B} & H_{\pa D\ra\pa B}
\end{array}\right)\left(\begin{array}{c}
\Psi_{\pa B}\\
W\Psi_{\pa D}
\end{array}\right)\\ \label{0415-5}
&=&Y_{\pa B}^{-1}\left(\begin{array}{cc}
T_{\pa B} & H_{\pa D\ra\pa B}
\end{array}\right)\left(\begin{array}{cc}
L_{\pa B}-I & M_{\pa D\ra\pa B}\\
L_{\pa B\ra\pa D} & M_{\pa D}
\end{array}\right)^{-1}\left(\begin{array}{c}
F/2\\
0
\end{array}\right)
\en
respectively, where
\ben
T_{\pa B}&:=&\left(\widetilde T_{\pa B}(\tilde t_i,\tilde t_j)\right)_{i,j=0,1,\cdots,2\tilde n-1},\\
\widetilde T_{\pa B}(\tilde t_i,\tilde t_j)&:=&T_j^{(\tilde n)}(\tilde t_i)+\frac{\pi}{\tilde n}\left[k_A(\tilde t_i,\tilde t_j)+k_2(\tilde t_i,\tilde t_j)-\frac1{2\pi}\right],\\
T_j^{(\tilde n)}(\tilde t)&:=&-\frac1{\tilde n}\sum_{m=1}^{\tilde n-1}m\cos m(\tilde t-\tilde t_j)-\frac12\cos\tilde n(\tilde t-\tilde t_j),\\
k_A(\tilde t,\tilde \tau)&:=&\begin{cases}
\frac1{2\pi}\frac{2(1-\cos(\tilde t-\tilde \tau))\tilde x'(\tilde t)\cdot\tilde x'(\tilde \tau)-\cos(\tilde t-\tilde \tau)|\tilde x(\tilde t)-\tilde x(\tilde \tau)|^2}{(1-\cos(\tilde t-\tilde \tau))|\tilde x(\tilde t)-\tilde x(\tilde \tau)|^2}, & \tilde t\neq\tilde \tau,\\
\frac1\pi\frac{\frac16\tilde x'(\tilde t)\cdot\tilde x'''(\tilde t)-\frac14|\tilde x''(\tilde t)|^2+\frac5{12}|\tilde x'(\tilde t)|^2}{|\tilde x'(\tilde t)|^2}, & \tilde t=\tilde \tau,
\end{cases}\\
k_2(\tilde t,\tilde \tau)&:=&\begin{cases}
-\frac2\pi\frac{\{[\tilde x(\tilde t)-\tilde x(\tilde \tau)]\cdot(\tilde x'_2(\tilde t),-\tilde x'_1(\tilde t))\}\{[\tilde x(\tilde t)-\tilde x(\tilde \tau)]\cdot(\tilde x'_2(\tilde\tau),-\tilde x'_1(\tilde\tau))\}}{|\tilde x(\tilde t)-\tilde x(\tilde \tau)|^4}, & \tilde t\neq\tilde \tau,\\
\frac1{2\pi}\frac{|\tilde x''(\tilde t)\cdot(\tilde x'_2(\tilde t),-\tilde x'_1(\tilde t))|^2}{|\tilde x'(\tilde t)|^4}, & \tilde t=\tilde \tau,\\
\end{cases}\\
H_{\pa D\ra\pa B}&:=&\left(\frac{\pi}n\widetilde H_{\pa D\ra\pa B}(\tilde t_i,\tau_j)\right)_{i=0,1,\cdots,2\tilde n-1,j=0,1,\cdots,2n-1},\\
H_{\pa D\ra\pa B}(\tilde t,\tau)&:=&\frac1\pi\frac{[x(\tau)-\tilde x(\tilde t)]\cdot(\tilde x'_2(\tilde t),-\tilde x'_1(\tilde t))}{|\tilde x(\tilde t)-x(\tau)|^2}|x'(\tau)|.
\enn
For the calculation of the hypersingular operator $T_{0,\pa B}$ we refer to \cite{Rath,Kress95}.

\section{Iteration method for a precise reconstruction}\label{sec_Newton}
\setcounter{equation}{0}

In numerical implementation, it is impossible to calculate $I_2(\Om)$ defined by (\ref{I2}) for all $\Om\subset\mathcal O$ with $\mathcal O$ introduced  in Remark \ref{rem1029-2}.
For a more precise result of recovering the shape of $D$, we will use Newton's iteration method with an appropriate initial guess based on the result of the method introduced in Section \ref{s3}.
We refer to \cite{BC2002,Kirsch93} for the details on iteration method in the case of scattering problem.

To introduce the iteration method, we consider the polygon $D$ with boundary $\pa D:=\{h(t)\in\R^2:t\in[0,2\pi)\}$, where $h$ is a $2\pi$-periodic function analogous to (\ref{1113}).
For any fixed boundary value $f\in H^s(\pa B)$ with $s\in[-\frac12,\frac12]$, we define the mapping $\mathcal F$ by
\be\label{1030-1}
\mathcal Fh=\pa_\nu u\quad\text{on }\pa B,
\en
where $u$ solves (\ref{1'})--(\ref{3'}).
By Remark \ref{0413-1} (iv), we know $\mathcal Fh\in H^{s-1}(\pa B)$.
The inverse problem is to solve (\ref{1030-1}), while the iteration method is to approximately solve (\ref{1030-1}) in a {\color{xxx}Newton type} iterative manner.
Precisely, for a proper initial guess {\color{xxx}$h=h_0$} we compute
\ben
h_{n+1}=h_n+q_n,\quad n=0,1,\cdots,
\enn
where $q_n$ solves the linearized equation of (\ref{1030-1}):
\be\label{1030-2}
\mathcal Fh_{n}+\mathcal F'_{h_n}q_{n}=\pa_\nu u\quad\text{on }\pa B.
\en
The domain derivative in (\ref{1030-2}) is defined by
\ben
\mathcal F'_hq=\lim_{t\ra0}\frac{\mathcal F(h+tq)-\mathcal F(h)}t.
\enn
Analogously to \cite[(28), (39) and Theorem 5.1]{BC2002}, the domain derivative is given by
\be\label{Frechet}
\mathcal F'_hq=\pa_\nu v|_{\pa B},
\en
where $v$ solves the following boundary value problem
\be\label{0415-2}
\Delta v=0 & \text{in }B\ba\ov{D},\\ \label{0415-3}
v=0 & \text{on }\pa B,\\ \label{0415-4}
v=-(\nu\cdot q)\pa_\nu u & \text{on }\pa D.
\en
In \cite[Corollary 4.2]{BC2002}, the existence of the shape derivative was justified for the Helmholtz equation in a proper function space near each corner point of $\partial D$. The arguments there can be easily adapted to the boundary value problem of  Laplace equation under consideration.

Due to the regularity of elliptic equations, {\color{hgh}the operator equation} (\ref{1030-2}) is ill-posed.
For a more stable numerical implementation, we may apply the Tikhonov regularization scheme to obtain
\be\label{211014-1}
q_n\approx(\alpha I+[\mathcal F'_{h_n}]^*\mathcal F'_{h_n})^{-1}[\mathcal F'_{h_n}]^*(\pa_\nu u|_{\pa B}-\mathcal Fh_{n}),
\en
where the regularization parameter $\alpha>0$ is appropriately chosen {\color{xxx}(see \cite[Section 4.5]{CK19})}.

It is difficult to calculate the values of $\pa_\nu u$ near the corners of $\pa D$ due to the singularities of elliptic boundary value problems in nonsmooth domains (see Theorem \ref{thm2.1} and \cite{Grisvard}).
To deal with this difficulty, we approximate the values of $\pa_\nu u$ near the corners of $\pa D$ in the following manner.
In view of Theorem \ref{0914}, we can easily deduce that
\ben
\pa_\nu u&=&\left(\begin{array}{cc}
T_{0,\pa B\ra\pa D} & K'_{0,\pa D}-I
\end{array}\right)\left(\begin{array}{c}
\varphi_{\pa B}\\
\varphi_{\pa D}
\end{array}\right)\quad\text{on }\pa D.
\enn
Therefore, {\color{xxx}$2|x'(t)|\pa_\nu u(x(t))$} can be approximated by
\be\label{0415-1}
Y_{\pa D}U=\left(\begin{array}{cc}
T_{\pa B\ra\pa D} & H_{\pa D}W-Y_{\pa D}
\end{array}\right)\left(\begin{array}{c}
\Psi_{\pa B}\\
\Psi_{\pa D}
\end{array}\right),
\en
where $Y_{\pa D}$ is given in \eqref{0507-1} and
\ben
U&:=&2\left(\pa_\nu u(x(t_0)),\pa_\nu u(x(t_1)),\cdots,\pa_\nu u(x(t_{2n-1}))\right)^\top,\\
T_{\pa B\ra\pa D}&:=&\left(\frac{\pi}{\tilde n}\widetilde T_{\pa B\ra\pa D}(t_i,\tilde\tau_j)\right)_{i=0,1,\cdots,2n-1,j=0,1,\cdots,2\tilde n-1},\\
\widetilde T_{\pa B\ra\pa D}(t,\tilde\tau)&:=&-\frac2\pi\frac{\{[x(t)-\tilde x(\tilde \tau)]\cdot(x'_2(t),-x'_1(t))\}\{[x(t)-\tilde x(\tilde \tau)]\cdot(\tilde x'_2(\tilde \tau),-\tilde x'_1(\tilde \tau))\}}{|x(t)-\tilde x(\tilde \tau)|^4}\\
&&+\frac1\pi\frac{(x'_2(t),-x'_1(t))\cdot(\tilde x'_2(\tilde \tau),-\tilde x'_1(\tilde \tau))}{|x(t)-\tilde x(\tilde \tau)|^2},\\
H_{\pa D}&:=&\left(\frac{\pi}n\widetilde H_{\pa D}(t_i,\tau_j)\right)_{i,j=0,1,\cdots,2n-1},\\
\widetilde H_{\pa D}(t,\tau)&:=&\begin{cases}
\frac{|x'(\tau)|}\pi\frac{[x(\tau)-x(t)]\cdot(x'_2(t),-x'_1(t))}{|x(t)-x(\tau)|^2}, & t\neq\tau,\\
\frac{x''(t)\cdot(x'_2(t),-x'_1(t))}{2\pi|x'(t)|}, & t=\tau.
\end{cases}
\enn
As pointed out in Remark \ref{rem0507}, $W\Psi_{\pa D}$ is viewed as an unknown vector to avoid the calculation of the inverse of $W$.
{\color{xxx}Applying} $W$ on both sides of \eqref{0415-1}, we obtain
\ben
Y_{\pa D}WU=\left(\begin{array}{cc}
WT_{\pa B\ra\pa D} & WH_{\pa D}-Y_{\pa D}
\end{array}\right)\left(\begin{array}{c}
\Psi_{\pa B}\\
W\Psi_{\pa D}
\end{array}\right),
\enn
where we have used the equlity $WY_{\pa D}=Y_{\pa D}W$.
In view of \eqref{1110-1}, $2w'(s)|x'(w(s))|\pa_\nu u(x(w(s)))$ on $\pa D$ can be approximated by
\ben
{\color{xxx}Y_{\pa D}WU}&=&\left(\begin{array}{cc}
WT_{\pa B\ra\pa D} & WH_{\pa D}-Y_{\pa D}
\end{array}\right)\left(\begin{array}{cc}
L_{\pa B}-I & M_{\pa D\ra\pa B}\\
L_{\pa B\ra\pa D} & M_{\pa D}
\end{array}\right)^{-1}\left(\begin{array}{c}
F\\
0
\end{array}\right).
\enn
Therefore, we may approximate $\pa_\nu u$ on $\pa D$ by
\be\label{0419}
{\color{xxx}U\approx(\alpha_0I+W)^{-1}(WU)=(\alpha_0I+W)^{-1}Y_{\pa D}^{-1}(Y_{\pa D}WU),}
\en
where the parameter $\alpha_0>0$ is appropriately chosen.

Analogously to Theorem \ref{0914}, we can easily deduce that
\ben
\pa_\nu v=\left(\begin{array}{cc}
T_{0,\pa B} & K'_{0,\pa D\ra\pa B}
\end{array}\right)\left(\begin{array}{cc}
K_{0,\pa B}-I & S_{0,\pa D\ra\pa B}\\
K_{0,\pa B\ra\pa D} & S_{0,\pa D}
\end{array}\right)^{-1}\left(\begin{array}{c}
0\\
-(\nu\cdot q)\pa_\nu u
\end{array}\right)\quad\text{on }\pa B,
\enn
where $v$ solves (\ref{0415-2})--(\ref{0415-4}).
The above equations can be approximated by
\ben
V=Y_{\pa B}^{-1}\left(\begin{array}{cc}
T_{\pa B} & H_{\pa D\ra\pa B}
\end{array}\right)\left(\begin{array}{cc}
L_{\pa B}-I & M_{\pa D\ra\pa B}\\
L_{\pa B\ra\pa D} & M_{\pa D}
\end{array}\right)^{-1}\left(\begin{array}{c}
0\\
-QU
\end{array}\right),
\enn
where $Y_{\pa B},T_{\pa B},H_{\pa D\ra\pa B},L_{\pa B},M_{\pa D\ra\pa B},L_{\pa B\ra\pa D},M_{\pa D}$ are given in \eqref{0415-5} and
\ben
V&:=&\left(\pa_\nu v(x(t_0)),\pa_\nu v(x(t_0)),\cdots,\pa_\nu v(x(t_{2n-1}))\right)^\top,\\
Q&:=&{\rm diag}\left(\nu(x(t_0))\cdot q(t_0),\nu(x(t_1))\cdot q(t_1),\cdots,\nu(x(t_{2n-1}))\cdot q(t_{2n-1})\right).
\enn
In view of \eqref{Frechet}, the domain derivative {\color{xxx}$\mathcal F'_hq=\pa_\nu v|_{\pa B}$} is thus approximated by $V$.

Noting that this paper focuses on polygon obstacles, we update the location of corners $\{P_\ell:\ell=1,\cdots,N\}$ of the polygon in (\ref{1113}) at each iteration step, instead of the coefficients of basis shape functions such as trigonometric polynomials in \cite{Kirsch93}.
Precisely, in the $m$-th iteration step the updated corners $\{P_\ell^{(m)}=(P_{\ell,1}^{(m)},P_{\ell,2}^{(m)}):{\color{xxx}\ell=1},\cdots,N\}$ are given by
\ben
\left(\begin{array}{c}
P_{1,1}^{(m)}\\
\vdots\\
P_{N,1}^{(m)}\\
P_{1,2}^{(m)}\\
\vdots\\
P_{N,2}^{(m)}
\end{array}\right)=\left(\begin{array}{c}
P_{1,1}^{(m-1)}\\
\vdots\\
P_{N,1}^{(m-1)}\\
P_{1,2}^{(m-1)}\\
\vdots\\
P_{N,2}^{(m-1)}
\end{array}\right)+\left(\begin{array}{c}
\Delta P_{1,1}^{(m)}\\
\vdots\\
\Delta P_{N,1}^{(m)}\\
\Delta P_{1,2}^{(m)}\\
\vdots\\
\Delta P_{N,2}^{(m)}
\end{array}\right),\quad m=1,2,\cdots,
\enn
where $\{P_\ell^{(m)}=(P_{\ell,1}^{(m)},P_{\ell,2}^{(m)}):\ell=1,\cdots,N\}$ are the corners in the $(m-1)$-th iteration step, and $\{\Delta P_\ell^{(m)}=(\Delta P_{\ell,1}^{(m)},\Delta P_{\ell,2}^{(m)}):\ell=1,\cdots,N\}$ are given by \eqref{211014-1} {\color{xxx}in terms of} \eqref{1113}.

\section{Numerical examples}\label{s4}
\setcounter{equation}{0}

In this section, we will display some numerical examples.

\begin{example}[Factorization method based on Dirichlet-to-Neumann operators]\label{FMDtN}
Let $B$ be a disk centered at the origin with radius $5$ (i.e., $\tilde x(\tilde t)=5(\cos\tilde t,\sin\tilde t)$ in (\ref{1110-5})).
The Dirichlet-to-Neumann operator $\Lambda_D$ is approximated by a $128\times128$ matrix.
The numerical results of indicator function $I(z)$ defined by \eqref{0409-1} for different obstacles are shown in Figure \ref{exFMDtN}, {\color{xxx}where the red solid line and black solid line representing the disk $B$ and the true obstacle $D$.}
The obstacle $D$ in Figure \ref{exFMDtN} (a) is a disk centered at $(2,3)$ with radius $0.5$.
The boundary $\pa D$ of obstacle $D$ in Figure \ref{exFMDtN} (b) is given by $x(t)=0.5(\cos t+0.65\cos 2t-0.65,1.5\sin t)+(2,1)$ for $0\leq t\leq2\pi$.
The obstacle $D$ in Figure \ref{exFMDtN} (c) is a polygon with corners given in order by $(0.25,-0.75)$, $(1.5,-0.5)$, $(1.5,0.5)$ and $(0.5,0.5)$.
\begin{figure}[!htbp]
  \centering
  \subfigure[]{
  \includegraphics[width=0.3\textwidth]{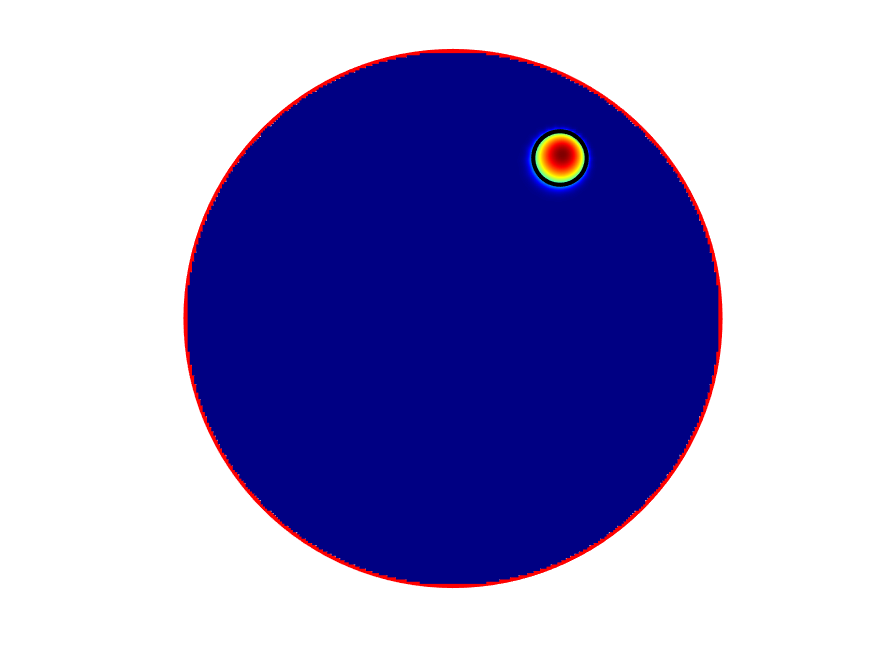}}
  \subfigure[]{
  \includegraphics[width=0.3\textwidth]{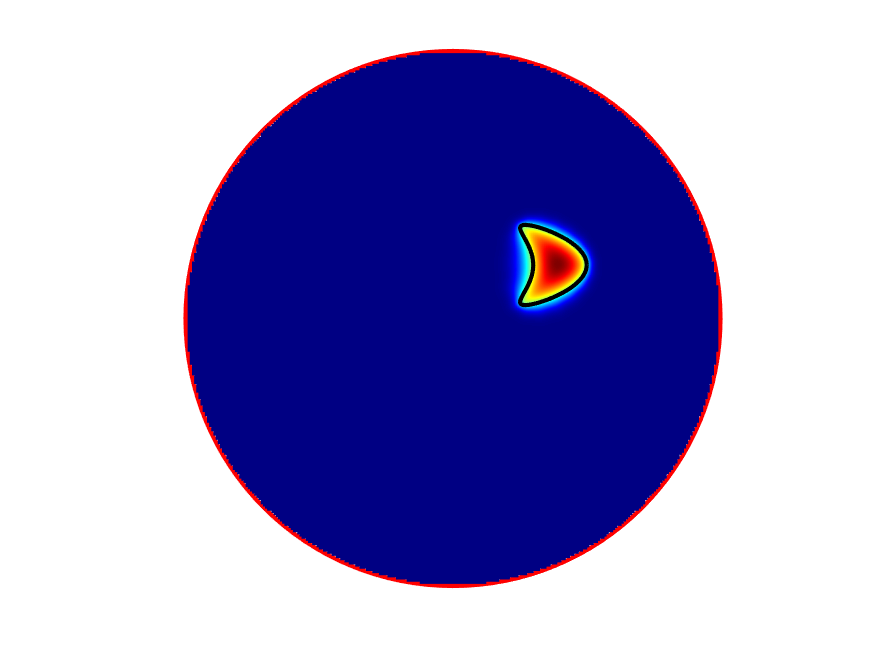}}
  \subfigure[]{
  \includegraphics[width=0.3\textwidth]{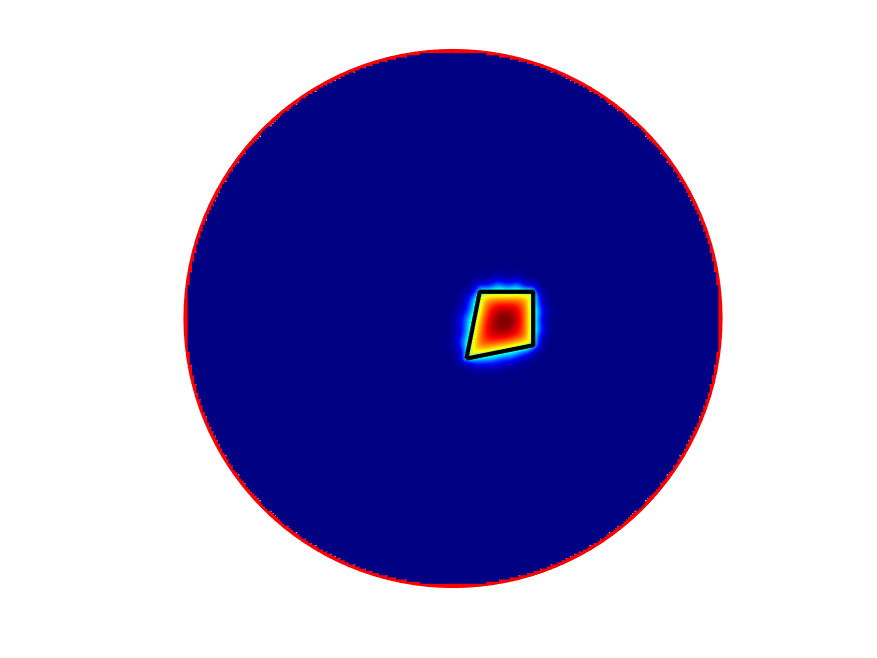}}
  \caption{Numerical results for Example \ref{FMDtN}.}\label{exFMDtN}
\end{figure}
\end{example}

\begin{example}[Determination of $\gamma$]\label{e1}
Let $B$ be a disk centered at the origin with radius $5$ and $D$ be a polygon with corners given in order by $(0.25,-0.75)$, $(1.5,-0.5)$, $(1.5,0.5)$ and $(0.5,0.5)$.
We set $\Om$ to be a disk centered at the origin with radius $3$.
The number of knots on $\pa B$ is set to be $64$, i.e., the Dirichlet-to-Neumann operator $\Lambda_D$ is approximated by a $64\times64$ matrix.
The noisy single pair of Cauchy data is given by $(f,\pa_\nu u_\delta)$ with $\pa_\nu u_\delta(x):=\pa_\nu u(x)(1+\delta\zeta)$.
Here, $\delta>0$ is the noise ratio and $\zeta$ is a uniformly distributed random number in $[-1,1]$.
The sampling knots for $\tau$ is set to be $\tau_\ell=\ell/20$ for $\ell=0,1,\!\cdots\!,40$.
\begin{figure}[!htbp]
  \centering
  \subfigure[$\gamma=e-2,\delta=0\%$]{
  \includegraphics[width=0.18\textwidth]{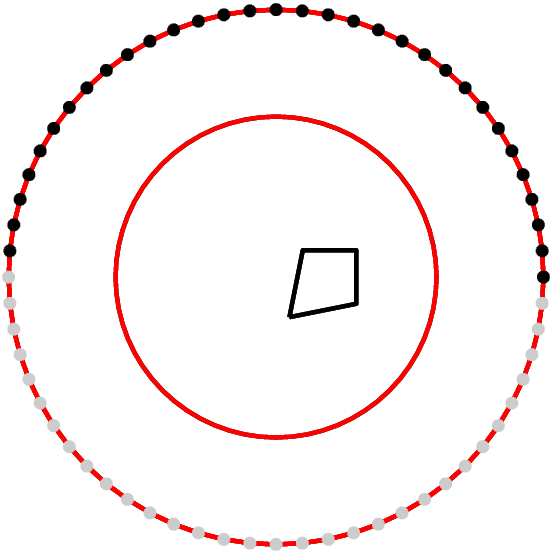}}
  \subfigure[$\gamma=e-2,\delta=1\%$]{
  \includegraphics[width=0.18\textwidth]{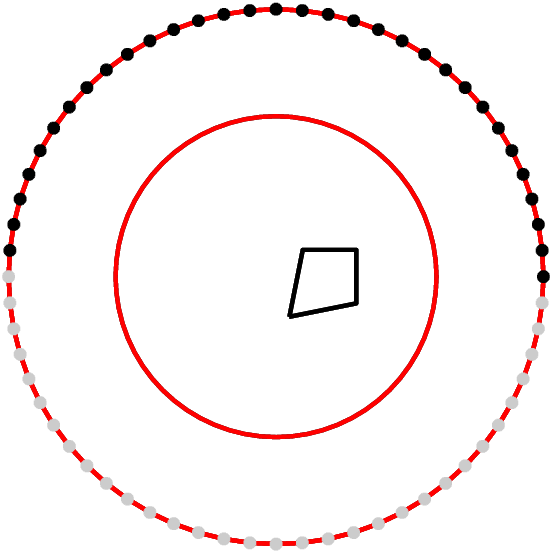}}
  \subfigure[$\gamma=1,\delta=1\%$]{
  \includegraphics[width=0.18\textwidth]{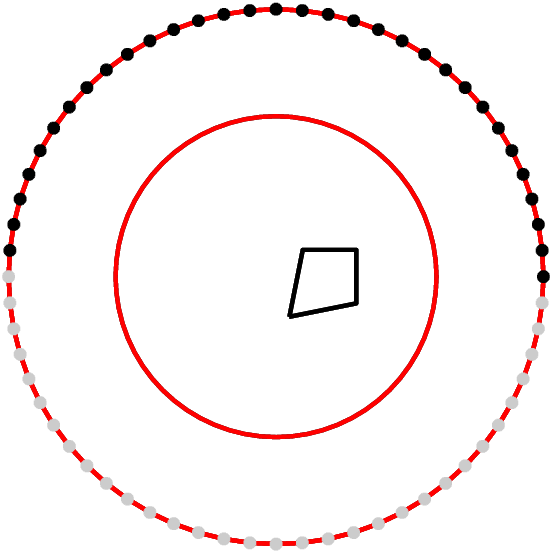}}
  \subfigure[$\gamma=\pi-2,\delta=1\%$]{
  \includegraphics[width=0.18\textwidth]{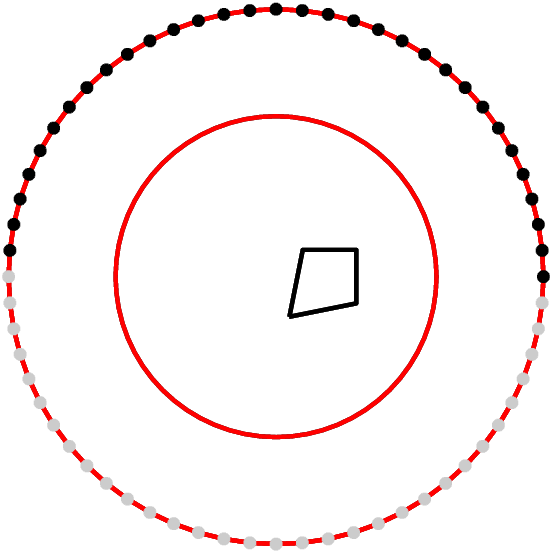}}
  \subfigure[$\gamma=\pi-2,\delta=0\%$]{
  \includegraphics[width=0.18\textwidth]{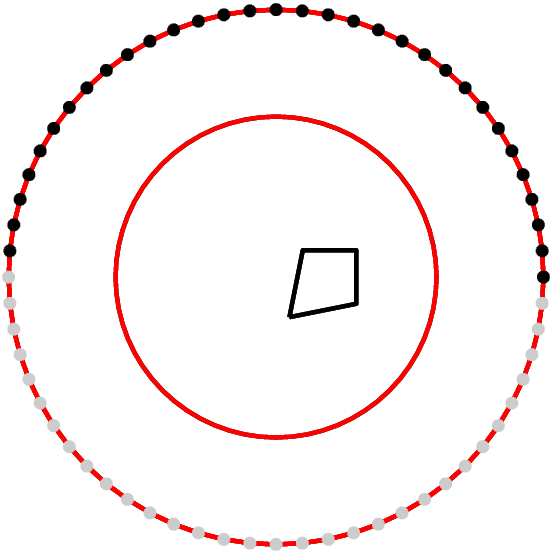}}\\
  \subfigure[$\gamma=e-2,\delta=0\%$]{
  \includegraphics[width=0.18\textwidth]{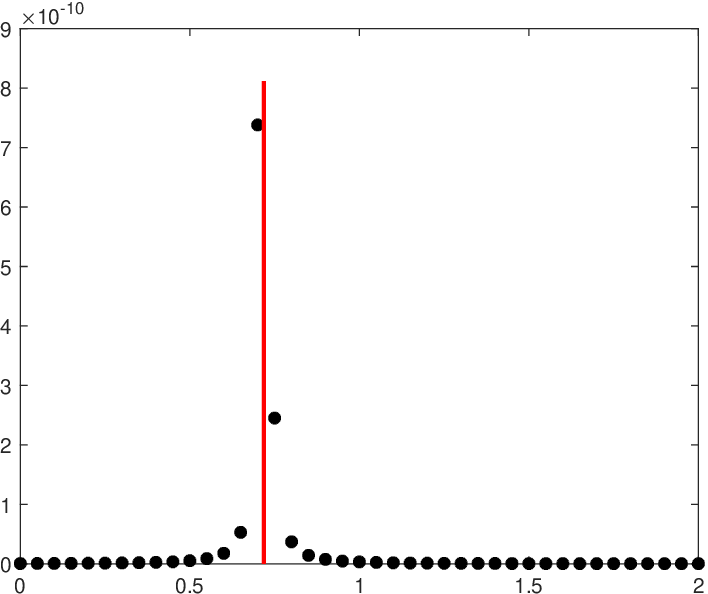}}
  \subfigure[$\gamma=e-2,\delta=1\%$]{
  \includegraphics[width=0.18\textwidth]{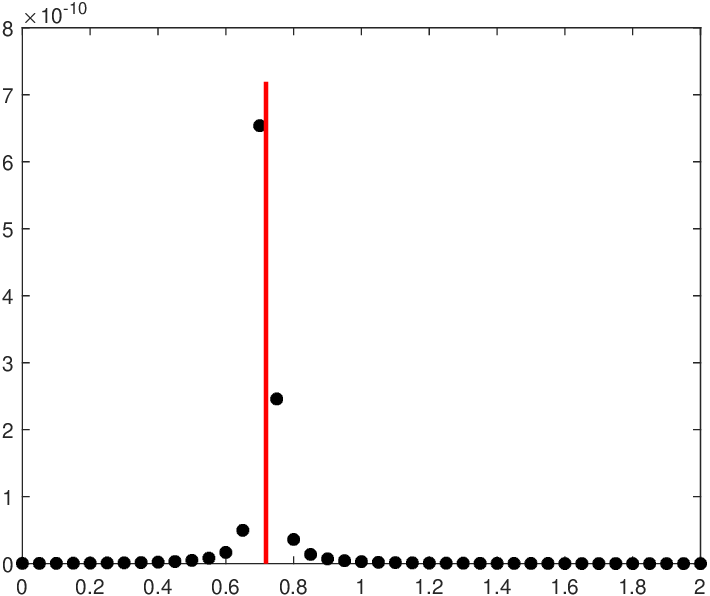}}
  \subfigure[$\gamma=1,\delta=1\%$]{
  \includegraphics[width=0.18\textwidth]{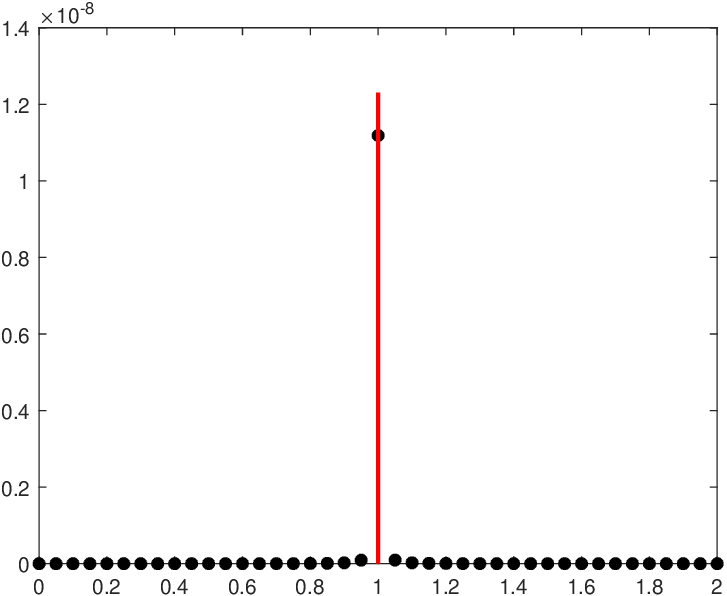}}
  \subfigure[$\gamma=\pi-2,\delta=1\%$]{
  \includegraphics[width=0.18\textwidth]{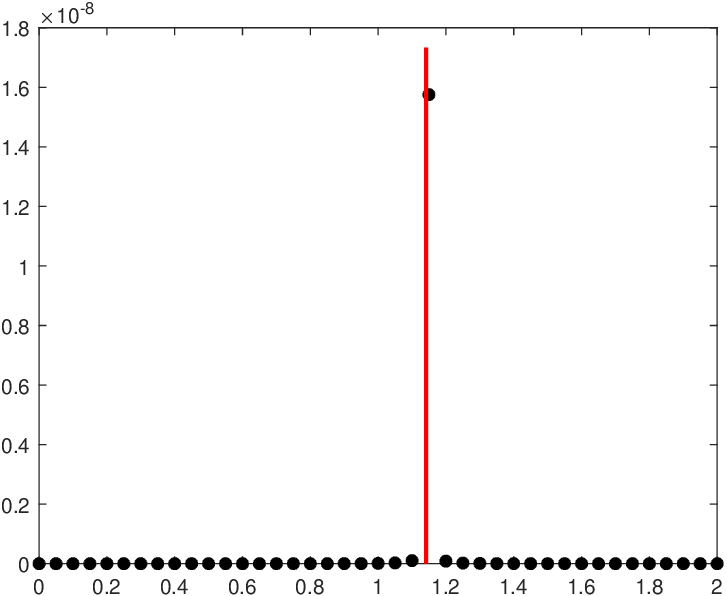}}
  \subfigure[$\gamma=\pi-2,\delta=0\%$]{
  \includegraphics[width=0.18\textwidth]{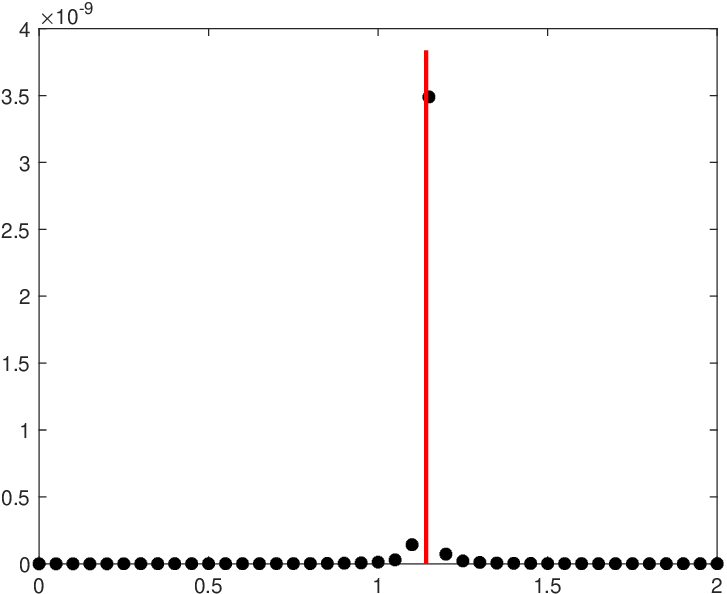}}
  \caption{Numerical results for Example \ref{e1} with different values of $\gamma$ and noise ratios.}\label{ex1}
\end{figure}
The numerical examples of $I_1(\tau)$ defined by (\ref{I1}) for different {\color{xxx} conductivity coefficients $\gamma$ and noise ratios $\delta$} with the same boundary value $f$ are shown in Figure \ref{ex1}.
The geometries of the problems are shown in Figure \ref{ex1} (a)--(e), where the black line represents $\pa D$, the small red circle represents $\pa\Omega$, and the large red circle represents $\pa B$. The boundary value $f$ takes value $0$ and $1$ at grey and black knots on $\pa B$, respectively.
The numerical results for $I_1(\tau)$ corresponding to Figure \ref{ex1} (a)--(e) are shown in Figure \ref{ex1} (f)--(j), respectively, where the red line represents the true value of $\gamma$ and black knots represent the points $(\tau_\ell,I_1(\tau_\ell))$, $\ell=0,1,\cdots,40$.
The numerical examples of $I_1(\tau)$ defined by (\ref{I1}) for different boundary values $f$ {\color{xxx}and noise ratios $\delta$} with the same {\color{xxx}conductivity coefficient} $\gamma=1$ are shown in Figure \ref{ex1b}.
The geometries of the problems are shown in Figure \ref{ex1b} (a)--(e), where the black line represents $\pa D$, the small red circle represents $\pa\Omega$, and the large red circle represents $\pa B$. The boundary value $f$ takes value $0,1,2$ at grey, black, blue knots on $\pa B$ in Figure \ref{ex1b} (a)--(c), respectively. The boundary value of Figure \ref{ex1b} (d), (e), (i) and (j) is {\color{xxx}given by} $f(x)=f(5(\cos\theta,\sin\theta))=\cos\theta$ on $\pa B$.
The numerical results for $I_1(\tau)$ corresponding to Figure \ref{ex1b} (a)--(e) are shown in Figure \ref{ex1b} (f)--(j), respectively, where the red line represents the true value of $\gamma$ and black knots represent the points $(\tau_\ell,I_1(\tau_\ell))$, $\ell=0,1,\cdots,40$.
\begin{figure}[!htbp]
  \centering
  \subfigure[$\delta=1\%$]{
  \includegraphics[width=0.18\textwidth]{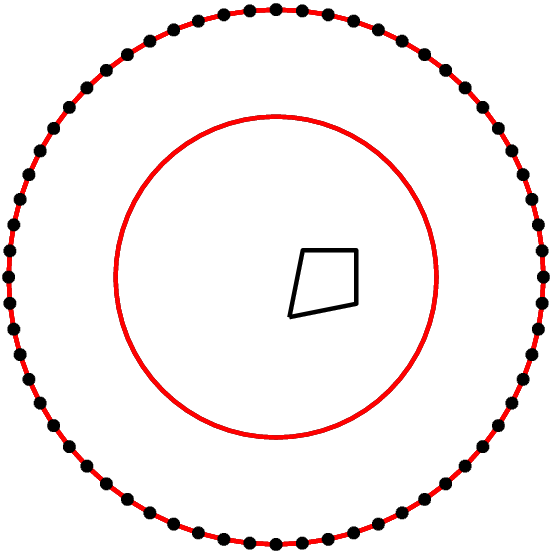}}
  \subfigure[$\delta=1\%$]{
  \includegraphics[width=0.18\textwidth]{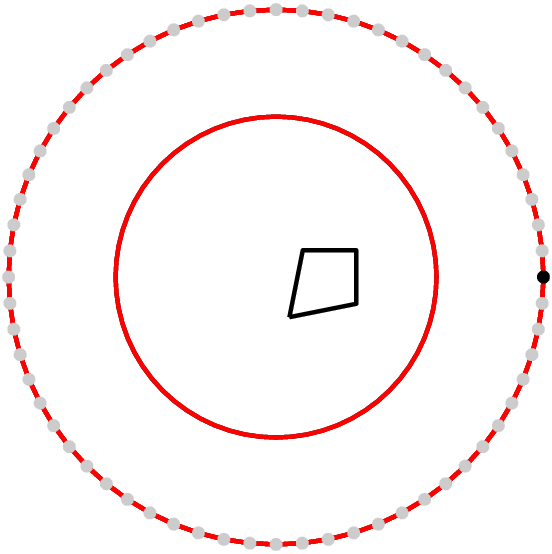}}
  \subfigure[$\delta=1\%$]{
  \includegraphics[width=0.18\textwidth]{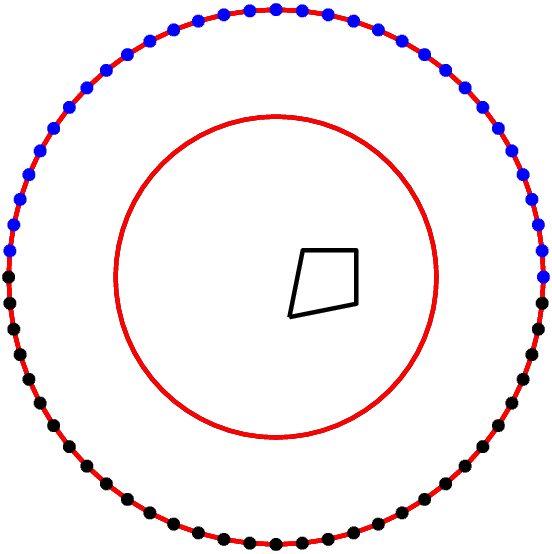}}
  \subfigure[$\delta=1\%$]{
  \includegraphics[width=0.18\textwidth]{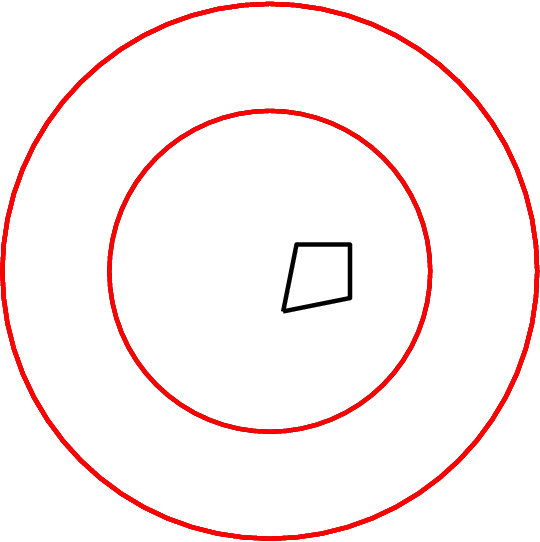}}
  \subfigure[$\delta=0\%$]{
  \includegraphics[width=0.18\textwidth]{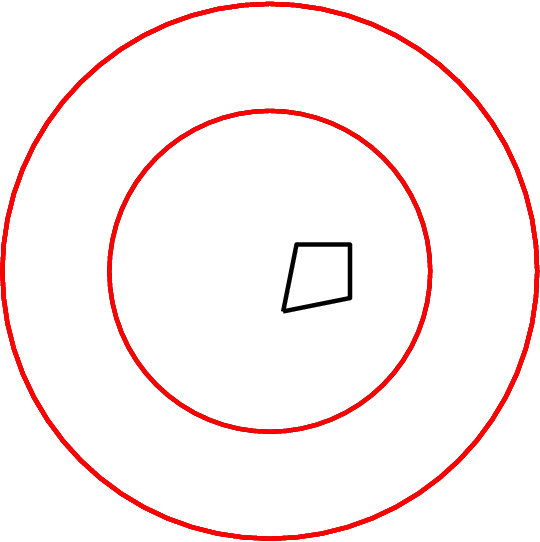}}\\
  \subfigure[$\delta=1\%$]{
  \includegraphics[width=0.18\textwidth]{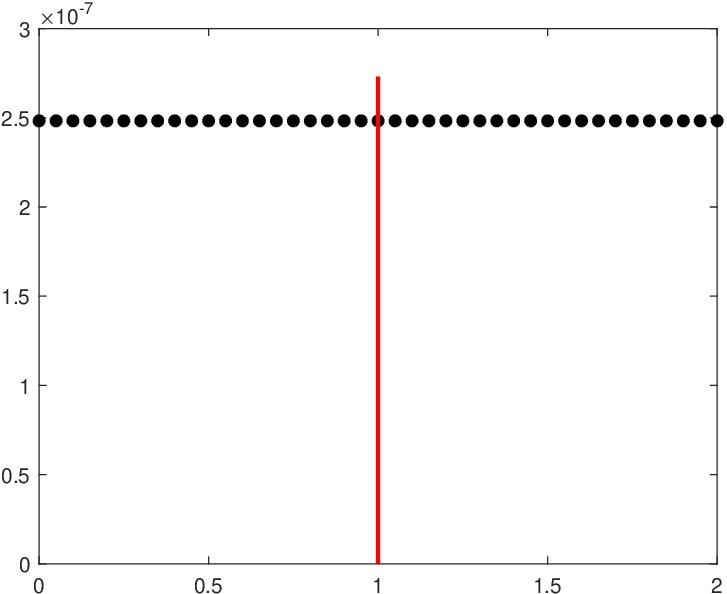}}
  \subfigure[$\delta=1\%$]{
  \includegraphics[width=0.18\textwidth]{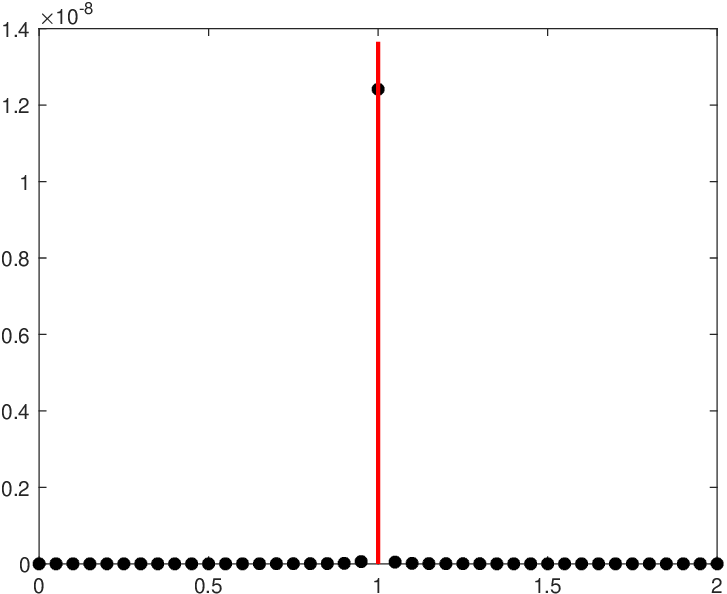}}
  \subfigure[$\delta=1\%$]{
  \includegraphics[width=0.18\textwidth]{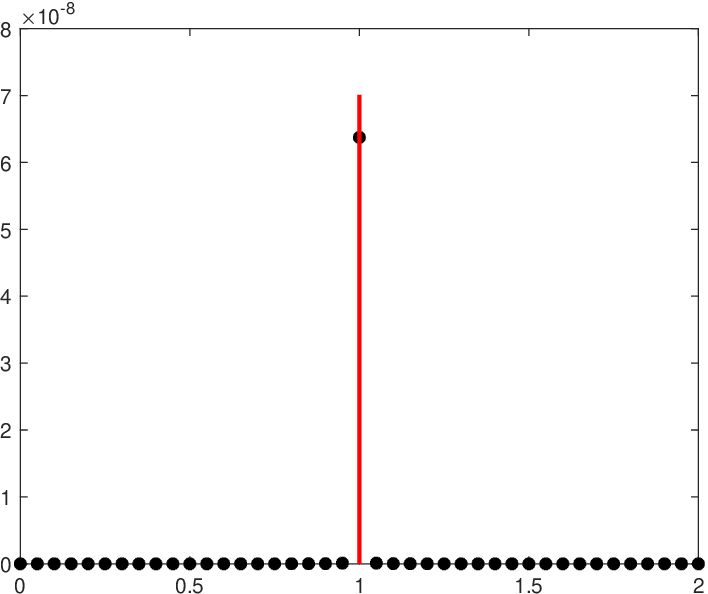}}
  \subfigure[$\delta=1\%$]{
  \includegraphics[width=0.18\textwidth]{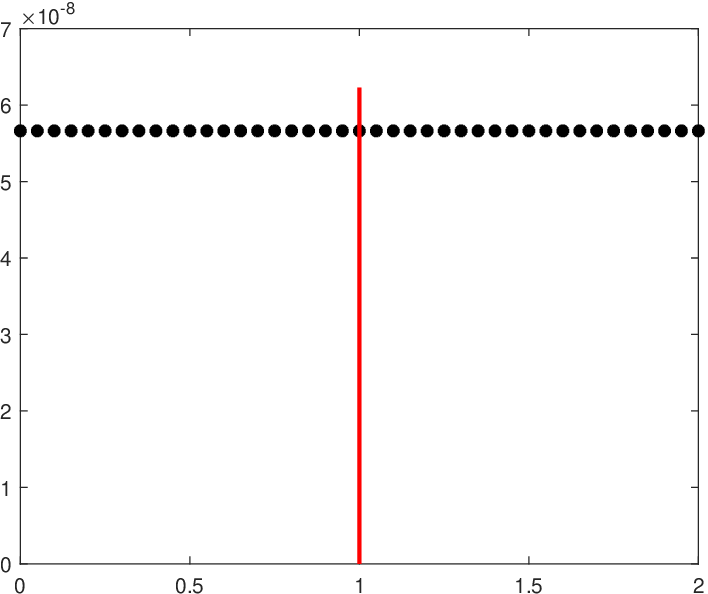}}
  \subfigure[$\delta=0\%$]{
  \includegraphics[width=0.18\textwidth]{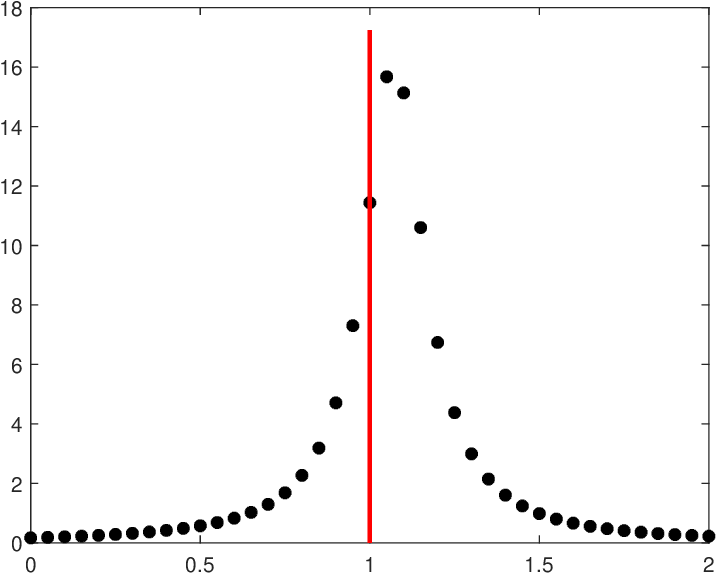}}
  \caption{Numerical results for Example \ref{e1} with different boundary values $f$ and noise ratios.}\label{ex1b}
\end{figure}
\end{example}

\begin{example}[Detection of location]\label{e2}
Let $D$ and $B$ be the same as in Example \ref{e1} and suppose that the {\color{xxx}conductivity coefficient $\gamma=1$} is known.
The boundary data $f$ and $\Lambda_Df$ are approximated by two $128\times1$ vectors.
We determine the location of $D$ in the following two different approaches.

\textbf{Approach 1.} Let $\Om_P^{(\ell)}$ be a disk centered at $P$ with radius $\ell/10$ for $\ell\!=\!5,6,\!\cdots\!,30$.
The number of knots on $\pa B$ is set to be $128$, i.e., the Dirichlet-to-Neumann operator $\Lambda_D$ is approximated by a $128\times128$ matrix.
To indicator the value of $I_2(\Om_P^{(\ell)})$ {\color{xxx}defined by} \eqref{I2}, we plot $\pa\Omega_P^{(\ell)}$ in the color given in terms of RGB values in $[0,1]\times[0,1]\times[0,1]$ as {\color{xxx}follows:}
\be\label{1029-3}
\begin{cases}
(V^{(\ell)},1-V^{(\ell)},0) & \text{if}\;V^{(\ell)}\ge0,\\
(0,1+V^{(\ell)},-V^{(\ell)}) & \text{if}\;V^{(\ell)}<0,
\end{cases}
\quad\text{ with }
V^{(\ell)}:=2\times\frac{I_2(\Om^{(\ell)}_P)-I_{min}}{I_{max}-I_{min}}-1,
\en
where $I_{max}:=\max_{\ell}\{I(\Om^{(\ell)})\}$ and $I_{min}:=\min_{\ell}\{I(\Om^{(\ell)})\}$.
The numerical results for $I_2(\Om_P^{(\ell)})$ defined by \eqref{I2} with different $P$ are shown in Figure \ref{ex2}, {\color{xxx}where the black line represents $\pa D$, the circle with knots represents $\pa B$, the boundary value $f$ takes value $0$ and $1$ at grey and black knots on $\pa B$, respectively, and the colored circles represent $\pa\Om_P^{(\ell)}$ for different $P$ and $\ell$ with its color indicating the value of $I_2(\Om_P^{(\ell)})$ in the sense of} (\ref{1029-3}).
From the numerical results shown in Figure \ref{ex2}, we see that the rough location of $D$ can be found from a single {\color{xxx}pair of} Cauchy data to (\ref{1})--(\ref{3}).
Moreover, we can imagine that Remark \ref{rem1029-2} can be numerically verified if the values of $I_2(\Om)$ for all possible domains $\Om$ in $\mathcal O$ are calculated.
\begin{figure}[!htbp]
  \centering
  \subfigure[$\Om_{(0,-1.5)}^{(\ell)}$]{
  \includegraphics[width=0.18\textwidth]{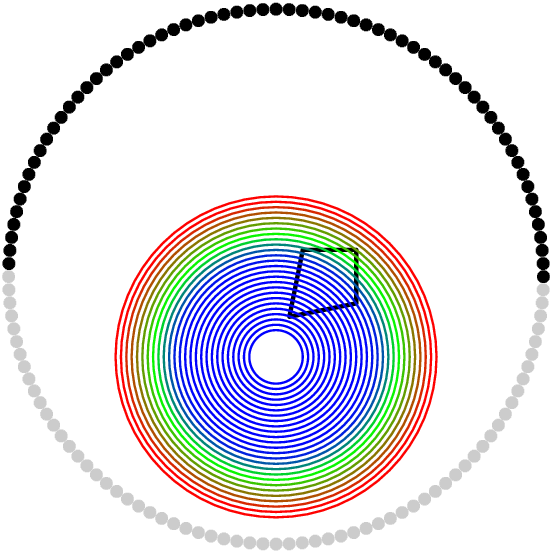}}
  \subfigure[$\Om_{(0,-0.5)}^{(\ell)}$]{
  \includegraphics[width=0.18\textwidth]{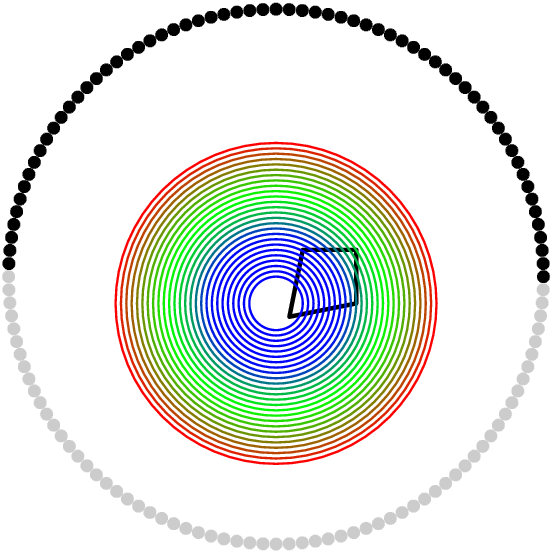}}
  \subfigure[$\Om_{(0,0)}^{(\ell)}$]{
  \includegraphics[width=0.18\textwidth]{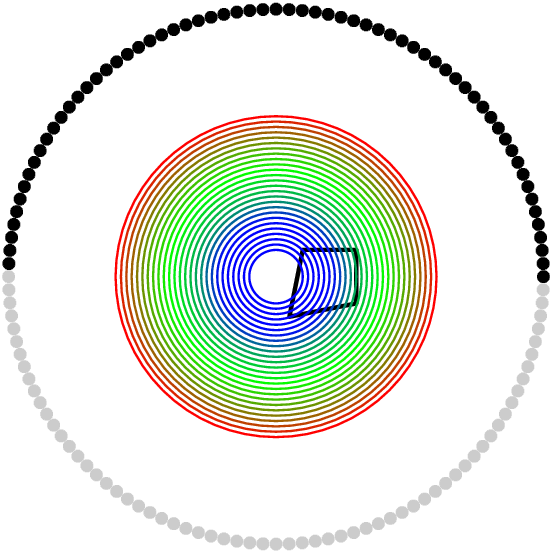}}
  \subfigure[$\Om_{(0,0.5)}^{(\ell)}$]{
  \includegraphics[width=0.18\textwidth]{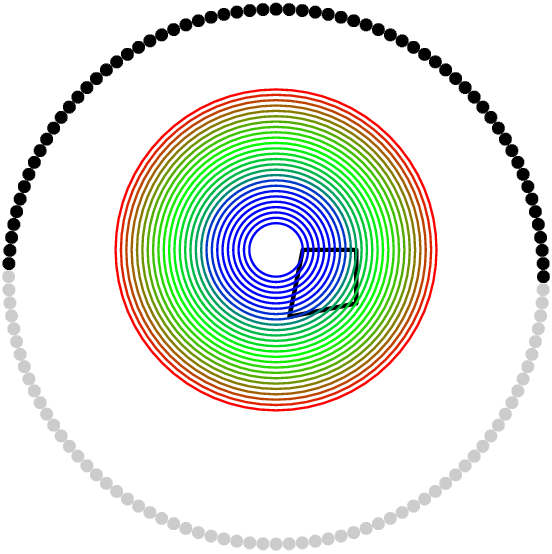}}
  \subfigure[$\Om_{(0,1.5)}^{(\ell)}$]{
  \includegraphics[width=0.18\textwidth]{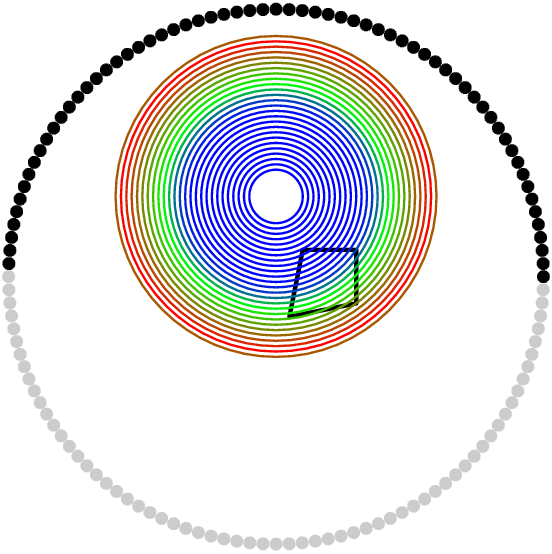}}
  \subfigure[$\Om_{(-1.5,0)}^{(\ell)}$]{
  \includegraphics[width=0.18\textwidth]{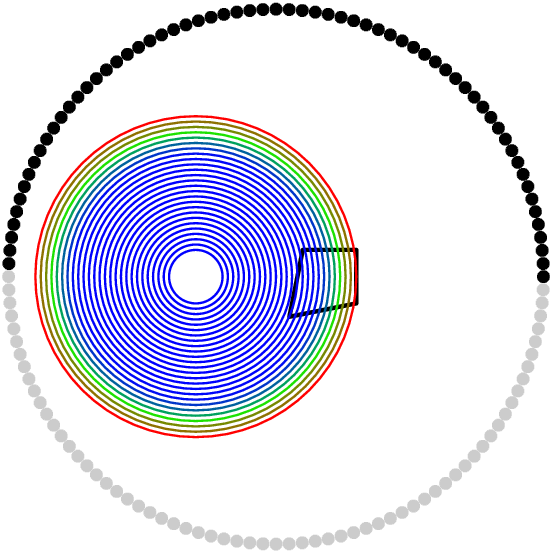}}
  \subfigure[$\Om_{(-0.5,0)}^{(\ell)}$]{
  \includegraphics[width=0.18\textwidth]{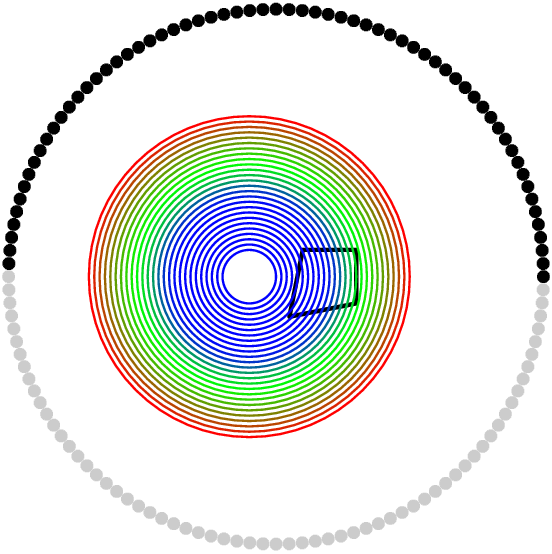}}
  \subfigure[$\Om_{(0,0)}^{(\ell)}$]{
  \includegraphics[width=0.18\textwidth]{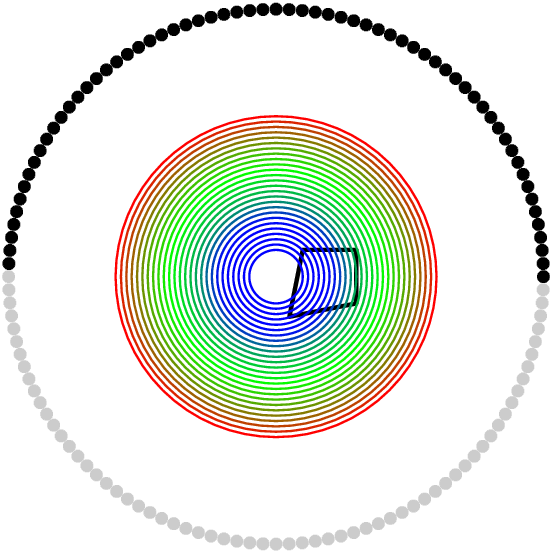}}
  \subfigure[$\Om_{(0.5,0)}^{(\ell)}$]{
  \includegraphics[width=0.18\textwidth]{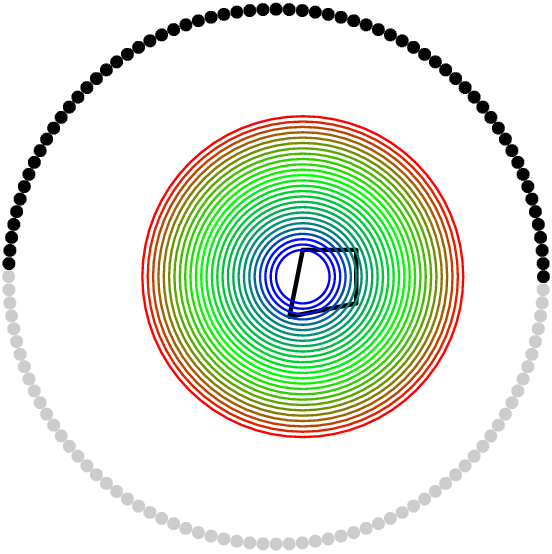}}
  \subfigure[$\Om_{(1.5,0)}^{(\ell)}$]{
  \includegraphics[width=0.18\textwidth]{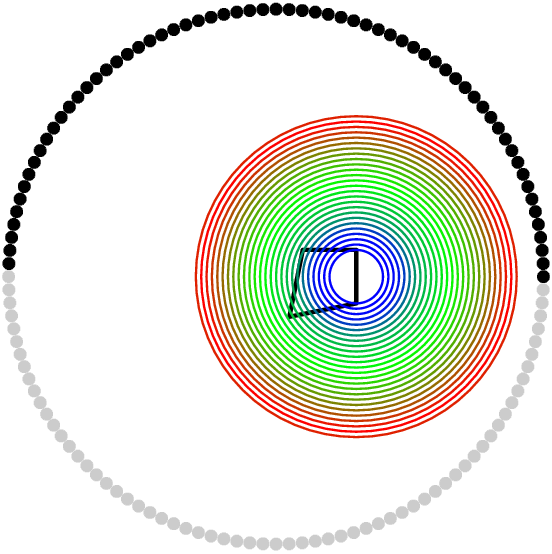}}
  \subfigure[$\Om_{(-1.5,0)}^{(\ell)}$]{
  \includegraphics[width=0.18\textwidth]{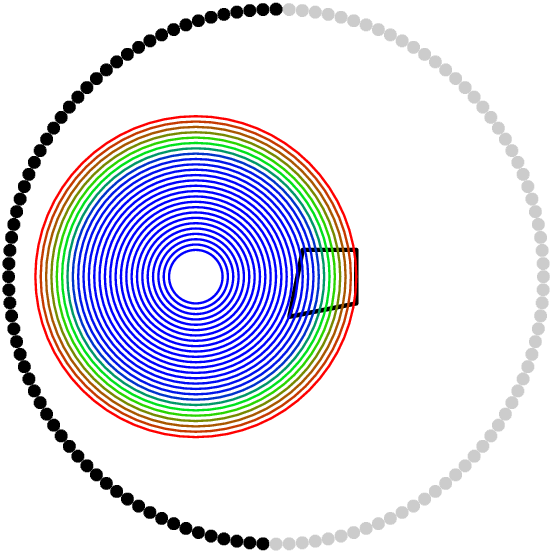}}
  \subfigure[$\Om_{(-0.5,0)}^{(\ell)}$]{
  \includegraphics[width=0.18\textwidth]{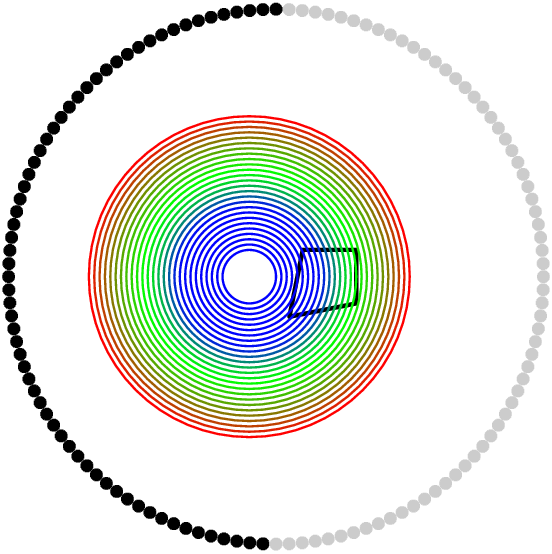}}
  \subfigure[$\Om_{(0,0)}^{(\ell)}$]{
  \includegraphics[width=0.18\textwidth]{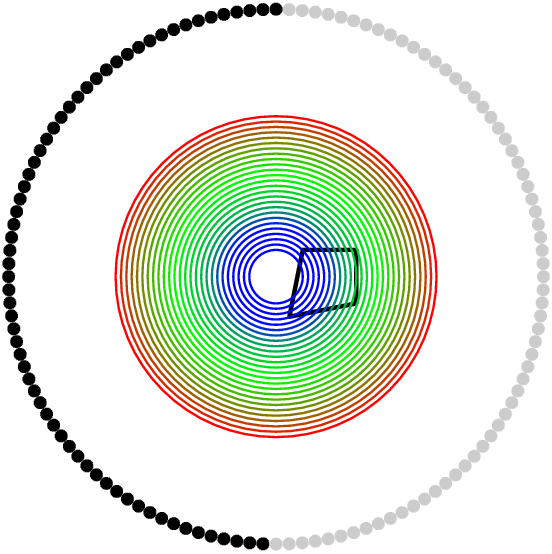}}
  \subfigure[$\Om_{(0.5,0)}^{(\ell)}$]{
  \includegraphics[width=0.18\textwidth]{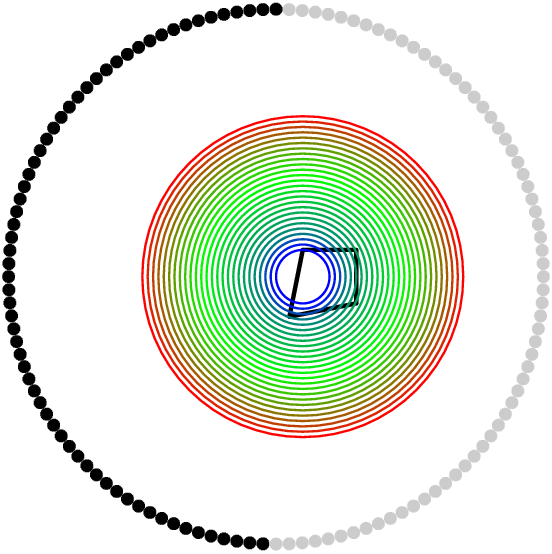}}
  \subfigure[$\Om_{(1.5,0)}^{(\ell)}$]{
  \includegraphics[width=0.18\textwidth]{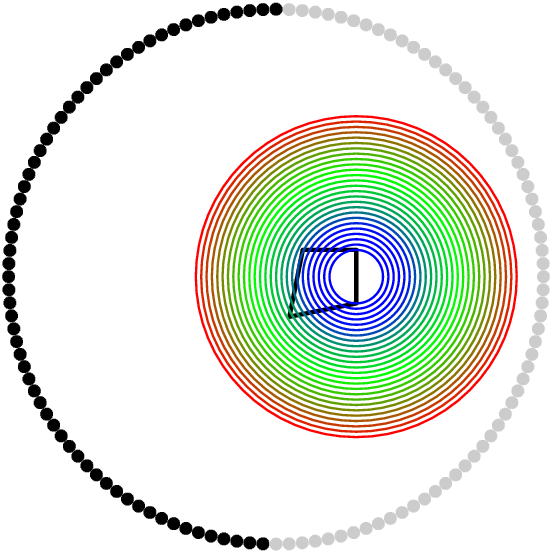}}
  \caption{Numerical results for Approach 1 of Example \ref{e2}.}\label{ex2}
\end{figure}

\textbf{Approach 2.} Let $r\in\{1,1/2,1/4,1/8\}$.
We set $\Om_{P}^{r}$ to be a disk centered at $P$ with radius $r$, where the center is set to be $P=P_{pq}=(-2+2pr,-2+2qr)$ for $p,q=0,1,\cdots,2/r$.
The numerical results for $I_2(\Om_{P_{pq}}^{r})$ defined by \eqref{I2} are shown in Figure \ref{ex3} (a)--(c), while the numerical results for $\ln(I_2(\Om_{P_{pq}}^{r}))$ are shown in Figure \ref{ex3} (d)--(e).
In Figure \ref{ex3}, {\color{xxx}the black line represents $\pa D$, the circle with knots represents $\pa B$, the boundary value $f$ takes value $0$ and $1$ at grey and black knots on $\pa B$, respectively.
The colored circles represent $\pa\Om_{P_{pq}}^{r}$ with its color indicating the value of $I_2(\Om_{P_{pq}}^{r})$ in the sense similar to} (\ref{1029-3}) in Figure \ref{ex3} (a)--(c), while the colored circles represent $\pa\Om_{P_{pq}}^{r}$ with its color indicating the value of $\ln(I_2(\Om_{P_{pq}}^{r}))$ in the sense similar to (\ref{1029-3}) in Figure \ref{ex3} (d)--(e).
\begin{figure}[!htbp]
  \centering
  \subfigure[$r=1$]{
  \includegraphics[width=0.18\textwidth]{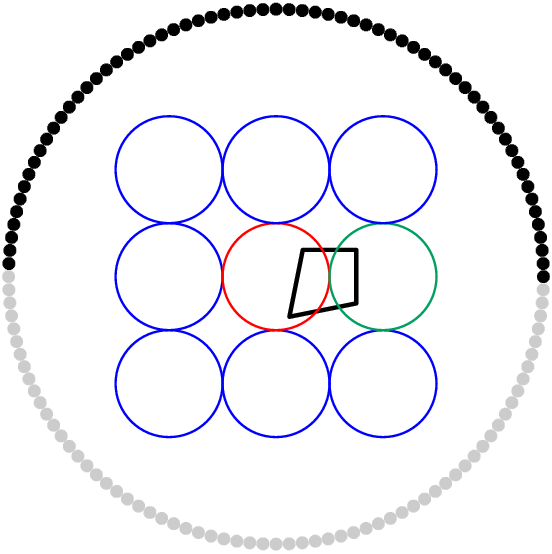}}
  \subfigure[$r=1/2$]{
  \includegraphics[width=0.18\textwidth]{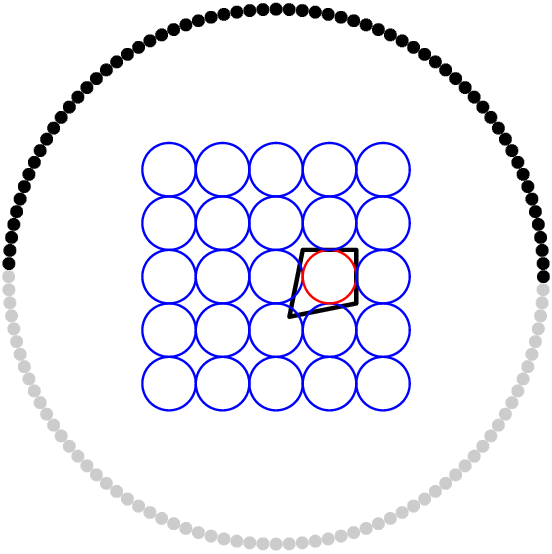}}
  \subfigure[$r=1/4$]{
  \includegraphics[width=0.18\textwidth]{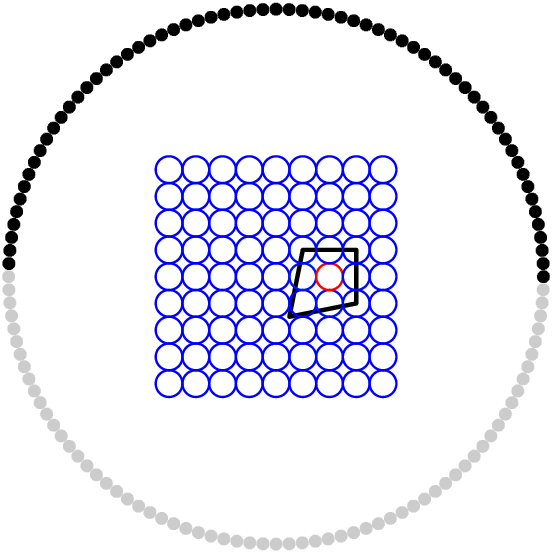}}
  \subfigure[$r=1/4$]{
  \includegraphics[width=0.18\textwidth]{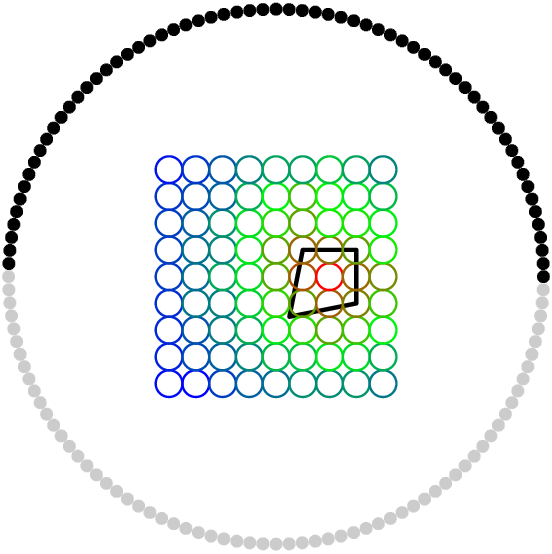}}
  \subfigure[$r=1/8$]{
  \includegraphics[width=0.18\textwidth]{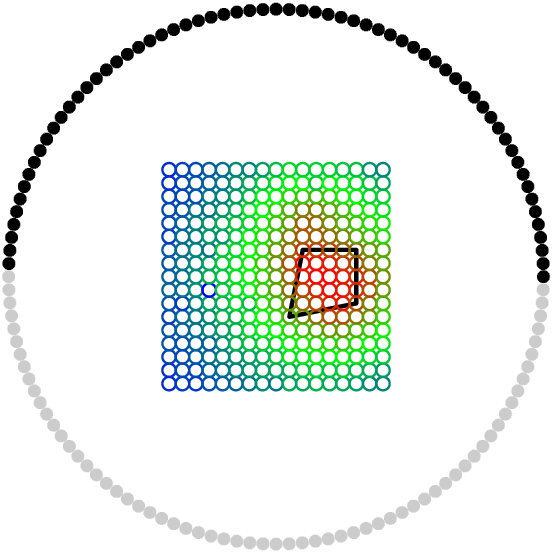}}
  \caption{Numerical results for Approach 2 of Example \ref{e2}.}\label{ex3}
\end{figure}

Note that $D$ is contained in none of the disks shown in Figure \ref{ex3}.
Nevertheless, it can be seen from Figure \ref{ex3} that the rough location of $D$ can be found in the way given by Approach 2 of Example \ref{e2}.
Intuitively, we think that in some sense the distance between the Cauchy data to $D$ and the gragh of the Dirichlet-to-Neumann operator to the disks close to $D$ may be nearer than the disks far away from $D$.
So far, we cannot explain why a rough location of $D$ can be identified {\color{xxx}by Approach 2 of} Example \ref{e2}.
\end{example}
\begin{example}[Iteration method]\label{e3}
{\color{xxx}Let $D$, $B$ and $\gamma$ be the same} as in Example \ref{e2}.
The boundary data $f$ and $\Lambda_Df$ are approximated by two $64\times1$ vectors.
\begin{figure}[!htbp]
  \centering
  \subfigure[$f=f_1$]{
  \includegraphics[width=0.23\textwidth]{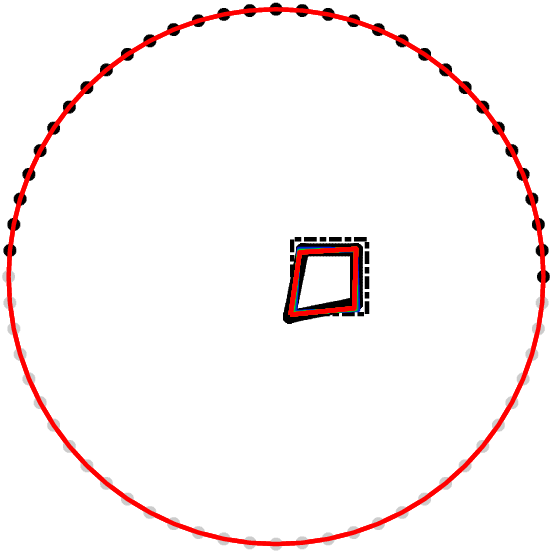}}
  \subfigure[$f=f_2$]{
  \includegraphics[width=0.23\textwidth]{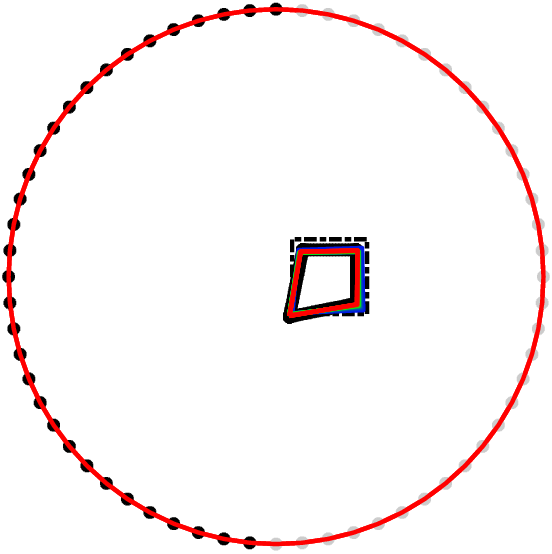}}
  \subfigure[$f=f_3$]{
  \includegraphics[width=0.23\textwidth]{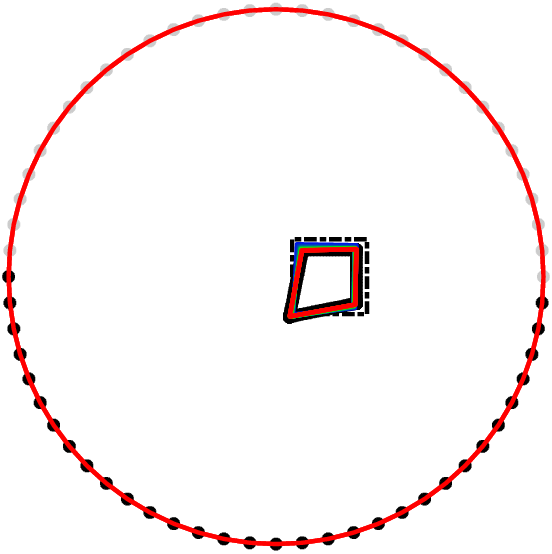}}
  \subfigure[$f=f_4$]{
  \includegraphics[width=0.23\textwidth]{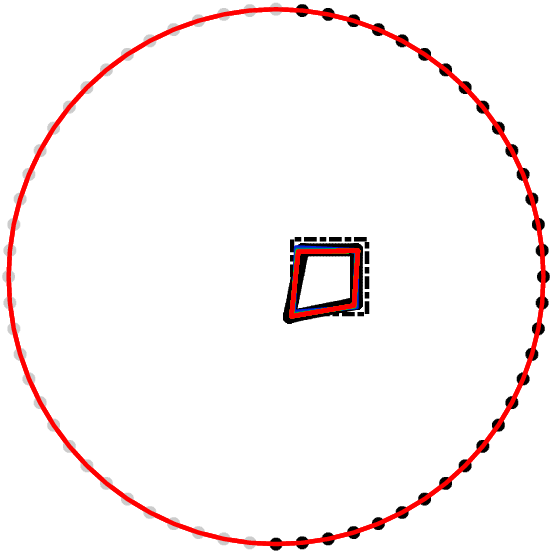}}
  \subfigure[$f=f_1$]{
  \includegraphics[width=0.23\textwidth]{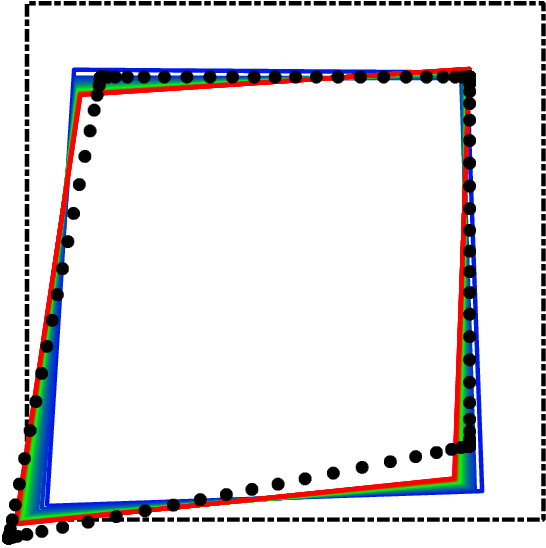}}
  \subfigure[$f=f_2$]{
  \includegraphics[width=0.23\textwidth]{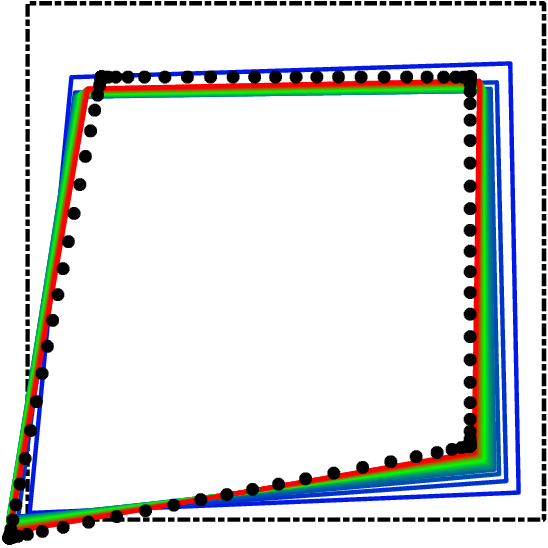}}
  \subfigure[$f=f_3$]{
  \includegraphics[width=0.23\textwidth]{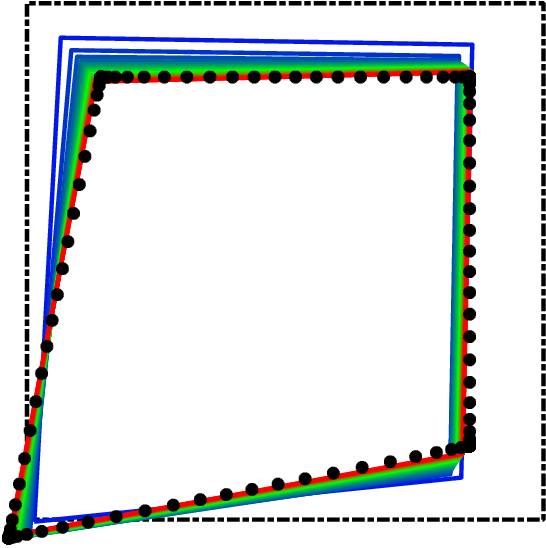}}
  \subfigure[$f=f_4$]{
  \includegraphics[width=0.23\textwidth]{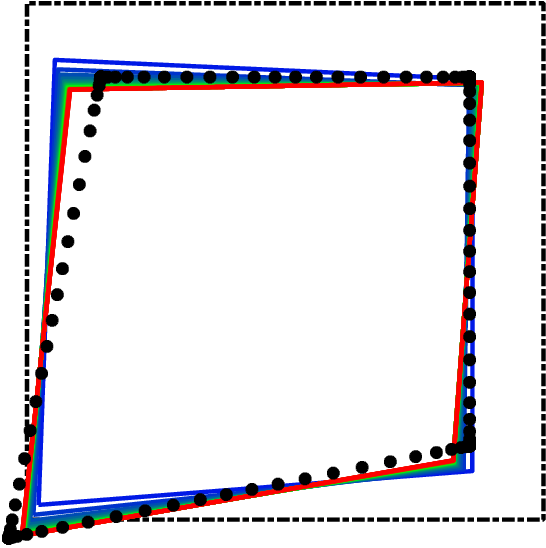}}
  \caption{Numerical results for Example \ref{e3} (i).}\label{iteration1}
\end{figure}

(i) The initial guess is given by the location of corners in order by $(0.3,-0.7)$, $(1.7,-0.7)$, $(1.7,0.7)$, $(0.3,0.7)$.
{\color{xxx}The total iteration number for each figure is $20$.}
The parameters in \eqref{211014-1} and \eqref{0419} are set to be $\alpha=10^{-3}$ and $\alpha_0=10^{-4}$, respectively.
{\color{xxx}The shape of the polygon} in the $m$-th iteration step will be plotted in the color given in terms of RGB values:
\be\label{1030-3}
\begin{cases}
(V^{(m)},1-V^{(m)},0) & \text{if}\;V^{(m)}\ge0,\\
(0,1+V^{(m)},-V^{(m)}) & \text{if}\;V^{(m)}<0,
\end{cases}
\quad\text{ with }
V^{(m)}:=\frac{2m}{\text{Total iteration number}}-1.
\en
The numerical results of iteration method based on a single pair of Cauchy data to different boundary values $f=f_j$, for $j=1,2,3,4$, on $\pa B$ are shown in Figure \ref{iteration1}, {\color{xxx}where the black dots represent $\pa D$, the red circle represents $\pa B$, the dashed line represents the initial guess, the boundary value $f=f_j$ takes value $0$ and $1$ at grey and black knots on $\pa B$, respectively.
Moreover, the colored lines represent the shape in each iteration step with its color indicating the current step number in the sense of} (\ref{1030-3}).

\begin{figure}[!htbp]
  \centering
  \subfigure[$5$ corners]{
  \includegraphics[width=0.23\textwidth]{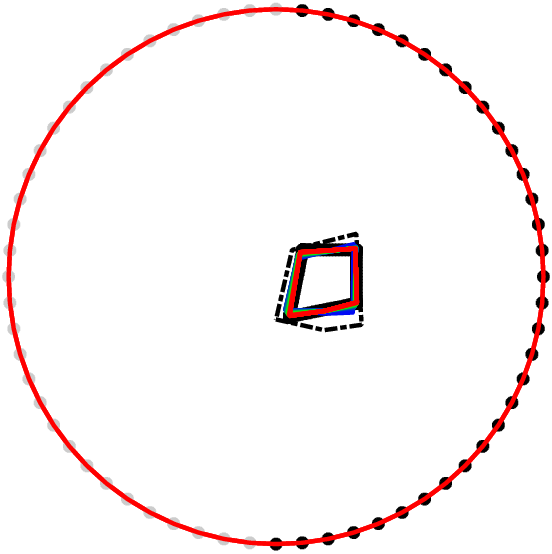}}
  \subfigure[$6$ corners]{
  \includegraphics[width=0.23\textwidth]{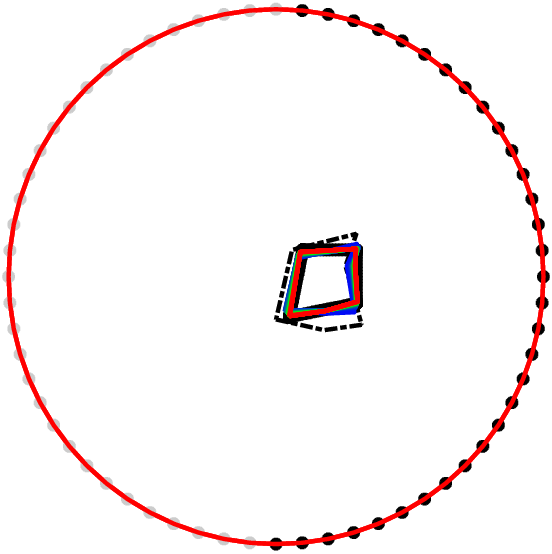}}
  \subfigure[$7$ corners]{
  \includegraphics[width=0.23\textwidth]{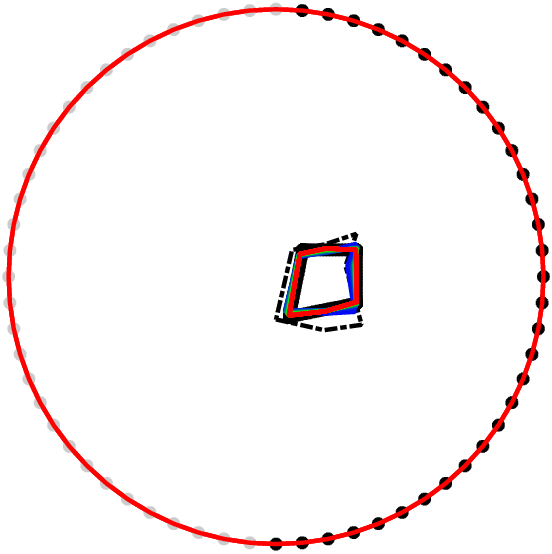}}
  \subfigure[$8$ corners]{
  \includegraphics[width=0.23\textwidth]{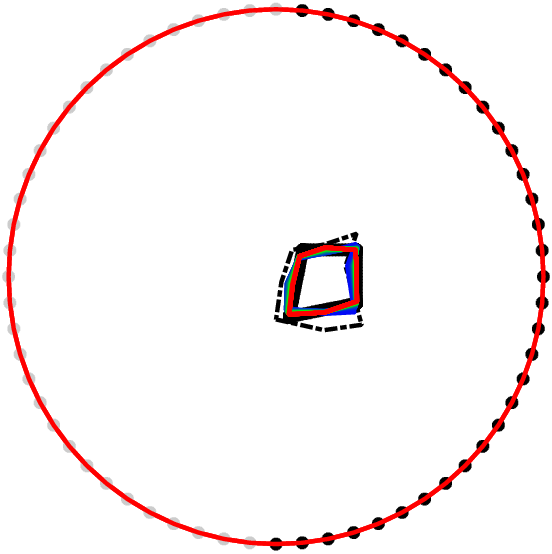}}
  \subfigure[$5$ corners]{
  \includegraphics[width=0.23\textwidth]{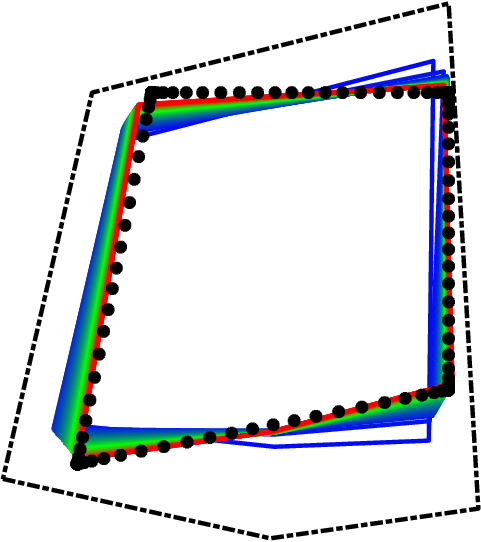}}
  \subfigure[$6$ corners]{
  \includegraphics[width=0.23\textwidth]{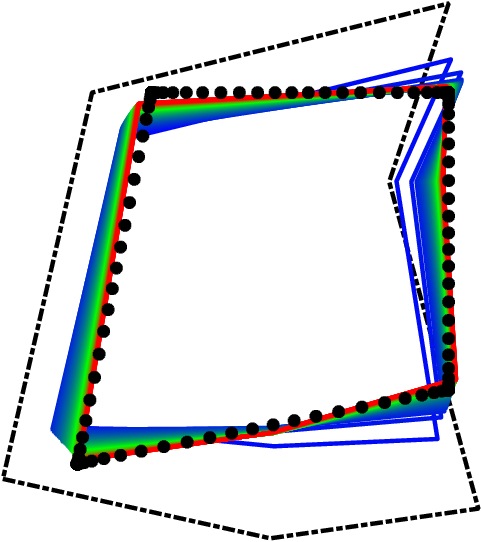}}
  \subfigure[$7$ corners]{
  \includegraphics[width=0.23\textwidth]{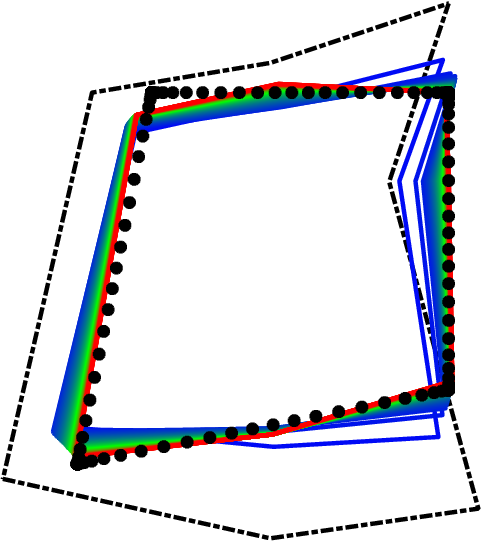}}
  \subfigure[$8$ corners]{
  \includegraphics[width=0.23\textwidth]{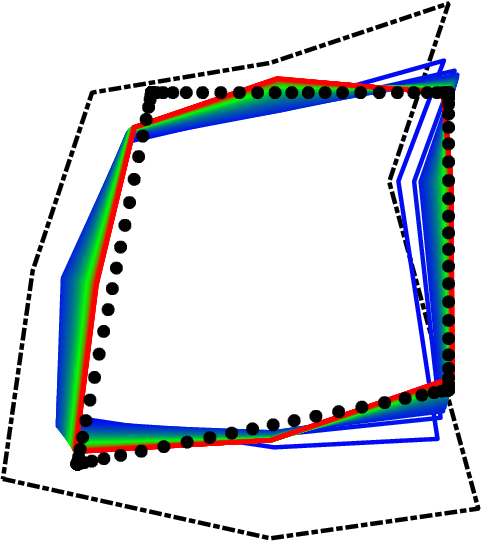}}
  \caption{Numerical results for Example \ref{e3} (ii).}\label{iteration2}
\end{figure}
(ii) {\color{xxx}Note} that we cannot take it for grant that the polygon $D$ has four corners. The numerical results of iteration method based on a single pair of Cauchy data to {\color{xxx}initial guesses of different corner numbers} are shown in Figure \ref{iteration2}, {\color{xxx}where the black dots represent $\pa D$, the red circle represents $\pa B$, the dashed line represents the initial guess, the boundary value $f$ takes value $0$ and $1$ at grey and black knots on $\pa B$, respectively, and the colored lines represent the shape in each iteration step with its color indicating the current step number in the sense of} (\ref{1030-3}).
The initial guess in Figure \ref{iteration2} (a) and (e) is given by the location of corners in order by $(0,-0.8)$, $(0.9,-1)$, $(1.6,-0.9)$, $(1.5,0.8)$, $(0.3,0.5)$ and corresponding number of knots on $\pa D$ is $130$.
The initial guess in Figure \ref{iteration2} (b) and (f) is given by $(0,-0.8)$, $(0.9,-1)$, $(1.6,-0.9)$, $(1.3,0.2)$, $(1.5,0.8)$, $(0.3,0.5)$ and corresponding number of knots on $\pa D$ is $132$.
The initial guess in Figure \ref{iteration2} (c) and (g) is given by $(0,-0.8)$, $(0.9,-1)$, $(1.6,-0.9)$, $(1.3,0.2)$, $(1.5,0.8)$, $(0.9,0.6)$, $(0.3,0.5)$ and corresponding number of knots on $\pa D$ is $196$.
The initial guess in Figure \ref{iteration2} (d) and (h) is given by $(0,-0.8)$, $(0.9,-1)$, $(1.6,-0.9)$, $(1.3,0.2)$, $(1.5,0.8)$, $(0.9,0.6)$, $(0.3,0.5)$, $(0.1,-0.1)$ and corresponding number of knots on $\pa D$ is $256$.
The total iteration {\color{xxx}number} for each figure are $50$.
{\color{xxx}The parameters in} \eqref{211014-1} and \eqref{0419} are set to be $\alpha=10^{-4}$ and $\alpha_0=10^{-5}$, respectively.
\end{example}

\section{Conclusion}\label{s5}
\setcounter{equation}{0}

In this paper, we have established the uniqueness of the inverse problem to determine the coefficient $\gamma$ and the Dirichlet polygon $D$ from a single pair of Cauchy data to the boundary value problem (\ref{1})--(\ref{3}) under some a priori assumption on the Dirichlet boundary value $f$ on $\pa B$.
{\color{hgh}A domain-defined sampling method was proposed to determine $\gamma$ and a modified factorization method was established to roughly recover the location and shape of $D$.
An iteration method is then used to improve the reconstruction of the geometry of $D$.}
All these numerical methods are based on a single pair of Cauchy data.
We note that similar uniqueness results can be established for Neumann obstacles and Robin obstacles with a constant Robin coefficient.
Similarly, the hybrid method proposed in this paper carries over to Neumann and Robin obstacles with minor modifications.
{\color{xxx}As an ongoing work}, we will extend the above results from the Laplace's equation to the Helmholtz equation.

\section*{Acknowledgements}

The authors would like to thank Professor Ikehata Masaru from Hiroshima University for helpful and stimulating discussions on the recovery of constant conductivity coefficient by a single pair of Cauchy data.
{\color{xxx}The work of Xiaoxu Xu is partially} supported by National Natural Science Foundation of China grant 12201489, {\color{xxx}Shaanxi Fundamental Science Research Project for Mathematics and Physics (Grant No.23JSQ025), Young Talent Fund of Association for Science and Technology in Shaanxi, China (Grant No.20240504),} the Young Talent Support Plan of Xi'an Jiaotong University, the Fundamental Research Funds for the Central Universities grant xzy012022009. The work of {\color{xxx}Guanghui Hu} is partially supported by National Natural Science Foundation of China grant 12071236 and the Fundamental Research Funds for Central Universities in China grant 63213025.

\end{document}